\documentclass{amsart}
\usepackage{amsmath}
\usepackage{amssymb}
\usepackage{amsthm}
\usepackage{dsfont}
\usepackage{lineno}
\usepackage{subfigure}
\usepackage{graphicx}
\usepackage{multirow}
\usepackage{color}
\usepackage{xspace}
\usepackage{pdfsync}
\usepackage{hyperref} 
\usepackage{algorithmicx}
\usepackage{algpseudocode}
\usepackage{algorithm}
\usepackage{mathtools}
\usepackage{enumitem}
\graphicspath{{MeshfreeFigures/}}

\DeclareMathOperator*{\argmin}{argmin}
\DeclareMathOperator*{\argmax}{argmax}

\DeclarePairedDelimiter\floor{\lfloor}{\rfloor}
\newcommand{\bq}{\begin{equation}}
\newcommand{\eq}{\end{equation}}
\newcommand{\R}{\mathbb{R}}
\newcommand{\Z}{\mathbb{Z}}
\newcommand{\abs}[1]{\left\vert#1\right\vert}

\newcommand{\G}{\mathcal{G}}
\newcommand{\Nf}{\mathcal{N}}
\newcommand{\bO}{\mathcal{O}}

\newcommand{\Dt}{\mathcal{D}}

\newcommand{\Af}{\mathcal{A}}

\newcommand{\Sf}{\mathcal{S}}
\newcommand{\MA}{Monge-Amp\`ere\xspace}

\algnewcommand{\LineComment}[1]{\State \(\triangleright\) #1}

\newtheorem{theorem}{Theorem}
\theoremstyle{lemma}
\newtheorem{lemma}[theorem]{Lemma}

\newtheorem{definition}[theorem]{Definition}

\newtheorem{remark}[theorem]{Remark}
\theoremstyle{remark}

\makeatletter
\newcommand\appendix@section[1]{%
\refstepcounter{section}%
\orig@section*{Appendix \@Alph\c@section: #1}%
}
\let\orig@section\section
\g@addto@macro\appendix{\let\section\appendix@section}
\makeatother

\begin{document}

\title[Meshfree finite difference methods]{Meshfree finite difference approximations for functions of the eigenvalues of the Hessian}

\author{Brittany D. Froese}
\thanks{This work was partially supported by NSF DMS-1619807.}
\address{Department of Mathematical Sciences, New Jersey Institute of Technology, University Heights, Newark, NJ 07102}
\email{bdfroese@njit.edu}

\begin{abstract}
We introduce meshfree finite difference methods for approximating nonlinear elliptic operators that depend on second directional derivatives or the eigenvalues of the Hessian.  Approximations are defined on unstructured point clouds, which allows for very complicated domains and a non-uniform distribution of discretisation points.  The schemes are monotone, which ensures that they converge to the viscosity solution of the underlying PDE as long as the equation has a comparison principle.  Numerical experiments demonstrate convergence for a variety of equations including problems posed on random point clouds, complex domains, degenerate equations, and singular solutions.
\end{abstract}

\date{\today}    
\maketitle

\section{Introduction}\label{sec:intro}

In this article we introduce meshfree finite difference methods for approximating a class of nonlinear elliptic partial differential equations (PDEs) that can be written in terms of second directional derivatives and/or the eigenvalues of the Hessian matrix in two dimensions.  This encompasses a number of important equations including Pucci minimal/maximal equations, a PDE for the convex envelope of a function, certain obstacle problems, and the \MA equation.  The methods are defined on unstructured point clouds, which allows for non-uniform distribution of discretisation points and complicated geometries.  As long as the PDE satisfies a comparison principle, these approximations are guaranteed to converge to the weak (viscosity) solution of the underlying PDE.

\subsection{Background}\label{sec:background}
Fully nonlinear elliptic partial differential equations (PDEs) arise in numerous applications including reflector/refractor design~\cite{GlimmOlikerReflectorDesign}, meteorology~\cite{Cullen}, differential geometry~\cite{Caf_MAGeom}, seismology~\cite{EFWass}, astrophysics~\cite{FrischUniv}, mesh generation~\cite{Budd}, computer graphics~\cite{Osher_book}, and mathematical finance~\cite{FlemingSoner}.  Because of the prevalence of these equations in applications, the development of robust numerical methods is a priority.  

In recent years, the numerical solution of these equations has received a great deal of attention, and several new methods have been developed including finite difference methods~\cite{BFO_MA,FinnGrid,Loeper,Saumier,SulmanWilliamsRussell}, finite element methods~\cite{Awanou,Bohmer,BrennerNeilanMA2D,Smears}, least squares methods~\cite{DGnum2006}, and methods involving fourth-order regularisation terms~\cite{FengNeilan}.  However, these methods are not designed to compute weak solutions. When the ellipticity of the equation is degenerate or no smooth solution exists, methods become very slow, are unstable, or converge to an incorrect solution.

Using a framework developed by Barles and Souganidis~\cite{BSnum}, provably convergent (monotone) methods have recently been constructed for several fully nonlinear equations~\cite{FO_MATheory,ObermanDiffSchemes,ObermanWS}.  Methods with a similar flavour have been developed by constructing Markov chain approximations for equations with a control interpretation~\cite{Bonnans_HJB,Kushner}.  However, these methods are typically defined on uniform Cartesian grids and do not lend themselves to adaptivity or complicated geometries.

State-of-the-art methods have recently been applied to problems in refractor design~\cite{FroeseOptics}, which involve the solution of a two-dimensional \MA equation with degeneracy.  In that setting, it was desirable to introduce large gradients into the data.  Non-monotone methods were found to be unstable, while monotone methods were restricted to Cartesian grids and could not effectively resolve the large gradients in the data.  In order to improve results in this and other applications, it is necessary to develop convergent, adaptive methods for solving fully nonlinear elliptic equations. 

\subsection{Contribution of this work}\label{sec:contribution}
This article introduces a framework for constructing convergent approximations of elliptic equations on unstructured point clouds, which is a first step towards the adaptive methods that are needed by applications.  We focus on two-dimensional equations that can be written in terms of various second-directional derivatives,
\bq\label{eq:PDE1}
F(x,u(x),u_{\theta\theta}(x);\theta\in\Af\subset[0,2\pi)) = 0,
\eq
where the admissible set $\Af$ is used to characterise a finite subset of all unit vectors in $\R^2$.

We also consider functions of the eigenvalues $\lambda_-(D^2u(x)) \leq \lambda_+(D^2u(x))$ of the Hessian matrix,
\bq\label{eq:PDE2}
F\left(x,u(x),\lambda_-(D^2u(x)),\lambda_+(D^2u(x))\right) = 0,
\eq
which can be written in terms of the minimum and maximum second directional derivative over all possible directions in~$\R^2$.

Equations~\eqref{eq:PDE1}-\eqref{eq:PDE2} encompass a large range of nonlinear elliptic equations.  For example, as long as the PDE operator is a convex function of the Hessian matrix, it can be expressed in the form of~\eqref{eq:PDE1}, though the precise details of this representation may be non-trivial~\cite[Proposition~5.3]{Evans_NonlinearElliptic}.

The key idea is to select stencils that align as closely as possible with the relevant direction $e_\theta = (\cos\theta,\sin\theta)$.  This can be accomplished by relying on a suitable search neighbourhood, which must be large relative to the spatial resolution of the point cloud.  The resulting schemes are not consistent in the traditional sense---in particular, they are not exact on quadratic functions.  However, the truncation error does vanish as the point cloud is refined, and the schemes are monotone.

We describe conditions on the point cloud that ensure the existence of an appropriate meshfree finite difference approximation.  Following the work of Barles and Souganidis~\cite{BSnum}, we prove that our methods converge to the viscosity solution of the PDE as long as the equation satisfies a comparison principle. 

Using the framework of filtered methods~\cite{FOFiltered}, these meshfree schemes also open up many new possibilities for designing higher-order, provably convergent numerical methods on general meshes or point clouds.

\subsection{Contents}\label{sec:contents}
In section~\ref{sec:weak}, we review viscosity solutions and a convergence framework for fully nonlinear elliptic equations.  In section~\ref{sec:meshfree}, we describe our new meshfree finite difference approximations and provide convergence proofs.  In section~\ref{sec:compute}, we present several computational examples that demonstrate the power of these new schemes.  In section~\ref{sec:conclusions}, we provide concluding remarks and discuss future work.

\section{Weak Solutions}\label{sec:weak}

One of the challenges associated with the approximation of fully nonlinear PDEs is the fact that classical (smooth) solutions may not exist.  It thus becomes necessary to interpret PDEs using some notion of weak solution, and the numerical methods that are used need to respect this notion of weak solution.  The most common concept of weak solution for this class of PDEs is the \emph{viscosity solution}, which involves transferring derivatives onto smooth test functions via a maximum principle argument~\cite{CIL}.

\subsection{Viscosity solutions}\label{sec:visc}

The PDEs we consider in this work belong to the class of degenerate elliptic equations,
\bq\label{eq:PDE} F(x,u,D^2u(x)) = 0, \quad x\in\Omega\subset\R^2.\eq
\begin{definition}[Degenerate elliptic]\label{def:elliptic}
The operator
$F:\Omega\times\R\times\Sf^2\to\R$
is \emph{degenerate elliptic} if 
\[ F(x,u,X) \leq F(x,v,Y) \]
whenever $u \leq v$ and $X \geq Y$.
\end{definition}

\begin{remark}
The PDE operators~\eqref{eq:PDE1},\eqref{eq:PDE2} that we consider in this work are degenerate elliptic if they are non-decreasing functions of their second argument ($u$) and non-increasing functions of all subsequent arguments (which involve second directional derivatives).
\end{remark}

Since degenerate elliptic equations need not have classical solutions, solutions need to be interpreted in a weak sense.  The numerical methods developed in this article are guided by the very powerful concept of the viscosity solution~\cite{CIL}.  Checking the definition of the viscosity solution requires checking the value of the PDE operator for smooth test functions lying above or below the semi-continuous envelopes of the candidate solution.

\begin{definition}[Upper and Lower Semi-Continuous Envelopes]\label{def:envelope}
The \emph{upper and lower semi-continuous envelopes} of a function $u(x)$ are defined, respectively, by
\[ u^*(x) = \limsup_{y\to x}u(y), \]
\[ u_*(x) = \liminf_{y\to x}u(y). \]
\end{definition}

\begin{definition}[Viscosity subsolution (supersolution)]\label{def:subsuper}
An upper (lower) semi-continuous function $u$ is a \emph{viscosity subsolution (supersolution)} of~\eqref{eq:PDE} if for every $\phi\in C^2(\bar{\Omega})$, whenever $u-\phi$ has a local maximum (minimum)  at $x \in \bar{\Omega}$, then
\[ 
F_*^{(*)}(x,u(x),D^2\phi(x)) \leq (\geq)  0 .
\]
\end{definition}
\begin{definition}[Viscosity solution]\label{def:viscosity}
A function $u$ is a \emph{viscosity solution} of~\eqref{eq:PDE} if $u^*$ is a subsolution and $u_*$ a supersolution.
\end{definition}

\begin{remark}
This definition also accounts for Dirichlet boundary conditions if the PDE operator is extended to the boundary as
\[ F(x,u(x),D^2\phi(x)) = u(x)-g(x), \quad x\in\partial\Omega. \]
This provides a weak interpretation of the boundary conditions, which can also allow for viscosity solutions that are discontinuous at the boundary.
\end{remark}

An important property of many elliptic equations is the comparison principle, which immediately implies uniqueness of the solution.
\begin{definition}[Comparison principle]\label{def:comparison}
A PDE has a \emph{comparison principle} if whenever $u$ is an upper semi-continuous subsolution and $v$ a lower semi-continuous supersolution of the equation, then $u \leq v$ on $\bar{\Omega}$.
\end{definition}

Proving this form of the comparison principle is highly non-trivial, with very few results available for general degenerate elliptic equations.  Instead, this is typically done on a case-by-case basis with techniques adapted to the PDE in question.  In some cases, where viscosity solutions are discontinuous, the comparison result must be relaxed so that $u \leq v$ only in the interior of the domain $\Omega$.  In that case, Theorem~\ref{thm:converge} is modified to guarantee convergence only at points in the interior of the domain.  Full details of these comparison results go well beyond the scope of the present article.  We refer to~\cite{FroeseGauss} for an example of a recent result that rigorously establishes interior comparison and convergence for the equation of prescribed Gaussian curvature.

\subsection{Approximation of viscosity solutions}\label{sec:approxVisc}

In order to construct convergent approximations of elliptic operators, we will rely on the framework provided by Barles and Souganidis~\cite{BSnum} and further developed by Oberman~\cite{ObermanDiffSchemes}.

We consider finite difference schemes that have the form
\bq\label{eq:approx} F^\epsilon(x,u(x),u(x)-u(\cdot)) = 0 \eq
where $\epsilon$ is a small parameter.

The convergence framework requires notions of consistency and monotonicity, which we define below.

\begin{definition}[Consistency]\label{def:consistency}
The scheme~\eqref{eq:approx} is \emph{consistent} with the equation~\eqref{eq:PDE} if for any smooth function $\phi$ and $x\in\bar{\Omega}$,
\[ \limsup_{\epsilon\to0^+,y\to x,\xi\to0} F^\epsilon(y,\phi(y)+\xi,\phi(y)-\phi(\cdot)) \leq F^*(x,\phi(x),\nabla\phi(x),D^2\phi(x)), 
\]
\[ \liminf_{\epsilon\to0^+,y\to x,\xi\to0} F^\epsilon(y,\phi(y)+\xi,\phi(y)-\phi(\cdot)) \geq F_*(x,\phi(x),\nabla\phi(x),D^2\phi(x)). \]
\end{definition}

\begin{definition}[Monotonicity]\label{def:monotonicity}
The scheme~\eqref{eq:approx} is monotone if $F^\epsilon$ is a non-decreasing function of its final two arguments.
\end{definition}

Schemes that satisfy these two properties respect the notion of the viscosity solution at the discrete level.  In particular, these schemes preserve the maximum principle and are guaranteed to converge to the solution of the underlying PDE.

\begin{theorem}[Convergence~\cite{ObermanDiffSchemes}]\label{thm:convergeVisc}
Let $u$ be the unique viscosity solution of the PDE~\eqref{eq:PDE}, where $F$ is a degenerate elliptic operator with a comparison principle.  Let the finite difference approximation $F^\epsilon$ be consistent and monotone and let $u^\epsilon$ be any solution of the scheme~\eqref{eq:approx}, with bounds independent of $\epsilon$.  Then $u^\epsilon$ converges uniformly to $u$ as $\epsilon\to0$.
\end{theorem}

We remark that the above theorem assumes existence of a bounded solution to the approximation scheme.  This is typically straightforward to show for a consistent, monotone approximation of a well-posed PDE, though the precise details can vary slightly and rely on available well-posedness theory for the PDE in question.  When the scheme is strictly monotone (proper), stability follows immediately from a discrete comparison principle as in~\cite[Theorem~8]{ObermanDiffSchemes}.  For more complicated equations, the result can be established by constructing smooth sub- and super-solutions of the PDE, which are also sub- and super-solutions of the approximation scheme due to consistency.  Application of the comparison principle then leads to existence of a bounded solution~\cite[Lemmas~35-36]{FroeseGauss}.

\subsection{Wide stencil schemes}\label{sec:WS}

In order to construct a convergent approximation of the PDE~\eqref{eq:PDE}, it is sufficient to design consistent and monotone approximation schemes for second directional derivatives of the form
\[ F_\theta^\epsilon(x,u(x),u(x)-u(\cdot)) \approx - \frac{\partial^2 u}{\partial e_\theta^2}. \]
These can then be substituted directly into the PDE operator $F$, which by assumption is a monotone function of these derivatives.

However, constructing monotone approximations of these operators is not straightforward.  In fact, results by Motzkin and Wasow~\cite{MotzkinWasow} and Kocan~\cite{Kocan} demonstrate that there are elliptic operators for which \emph{no} bounded finite difference stencil will enable the construction of a consistent, monotone approximation.

Oberman~\cite{ObermanWS} addressed this issue by introducing the notion of \emph{wide stencil} finite difference schemes.  These schemes use centred difference approximations of the form
\bq\label{eq:WS} 
 \frac{\partial^2 u}{\partial e_\theta^2} = \frac{u(x+h_\theta e_\theta)+u(x-h_\theta e_\theta)-2u(x)}{h_\theta^2} + \bO(h_\theta^2)
\eq
for directions $e_\theta$ that align with the grid.  That is, there should exists some $h_\theta$ such that $h_\theta e_\theta=(m,n)$ where $m,n\in\Z$.  These finite difference approximations cannot simply rely on nearest neighbours; instead, they require wide stencils.  As stencils are allowed to grow wider, more directions can be accommodated.  

A fixed stencil width will only permit the discretisation of second derivatives in finitely many directions.  If these approximations are used for general elliptic operators of the form~\eqref{eq:PDE}, they will introduce additional discretisation error of the form $d\theta$, which corresponds to the size of angles that can be resolved on the stencil.  See Figure~\ref{fig:WS}.

\begin{figure}
\centering
\includegraphics[width=0.55\textwidth]{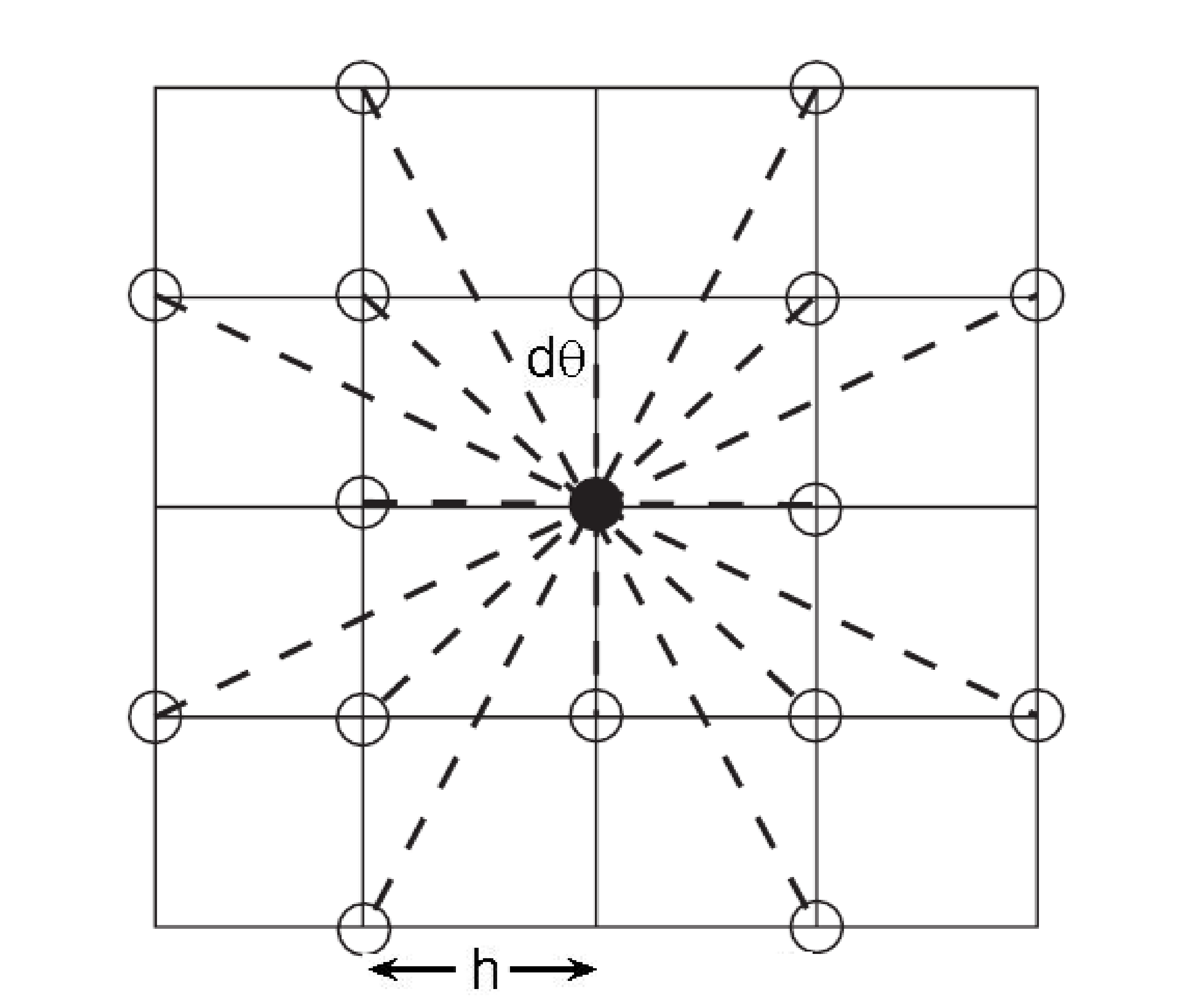}
\caption{A wide finite difference stencil.}
\label{fig:WS}
\end{figure}

While provably convergent wide stencil finite approximations can be constructed for nonlinear elliptic equations of the form~\eqref{eq:PDE}, they suffer from several limitations.  One restriction is that these approximations are defined only on uniform Cartesian grids, and do not extend naturally to non-uniform grids or non-rectangular domains.  A second problem with wide stencil schemes is the challenge of dealing with discretisation points near the boundary, where it is not possible to construct a wide stencil lying inside the domain.  One option is to use an (inconsistent) narrower stencil near the boundary and accept the resulting boundary layer in the computed solution.   In some cases, it is possible to use an altered scheme near the boundary, which is typically complicated and highly dependent on the particular form of the boundary conditions.

\section{Meshfree Finite Difference Approximations}\label{sec:meshfree}

In this section, we introduce a framework for constructing monotone approximations of second directional derivatives on general point clouds.  These approximations apply easily to complicated geometries and non-uniform distribution of discretisation points.  We describe the approximations, and also provide conditions on the point clouds that guarantee a convergent numerical method.

We focus the discussion on monotone approximation of second directional derivatives; these can then be used to approximate other nonlinear operators as described in subsection~\ref{sec:convergence}.

\subsection{Notation}\label{sec:notation}

We introduce the following notation.
\begin{itemize}
\item $\Omega\subset\R^2$ is a bounded domain with Lipschitz boundary $\partial\Omega$. 
\item $\G\subset\bar{\Omega}$ is a point cloud consisting of the points $x_i$, $i=1,\ldots,N$.
\item $h = \sup\limits_{x\in{\Omega}}\min\limits_{y\in\G}\abs{x-y}$ is the spatial resolution of the point cloud.  In particular, every ball of radius $h$ contained in $\bar{\Omega}$ contains at least one discretisation point $x_i$.
\item $h_B = \sup\limits_{x\in{\partial\Omega}}\min\limits_{y\in\G\cap\partial\Omega}\abs{x-y}$ is the resolution of the point cloud on the boundary.  In particular, every ball of radius $h_B$ centred at a boundary point $x\in\partial\Omega$ contains at least one discretisation point $x_i \in \G\cap\partial\Omega$ on the boundary.
\item $\delta = \min\limits_{x\in\Omega\cap\G}\inf\limits_{y\in\partial\Omega}\abs{x-y}$ is the distance between the set of interior discretisation points and the set of boundary discretisation points.  In particular, if $x_i\in\G\cap\Omega$ and $x_j\in\G\cap\partial\Omega$, then the distance between $x_i$ and $x_j$ is at least $\delta$.
\item $d\theta$ is the desired angular resolution of the meshfree finite difference approximation.
\item $r\equiv h(1+\sin(d\theta/2)+\cos(d\theta/2)\cot(d\theta/2))$ is the search radius associated with the point cloud.
\end{itemize}

\subsection{Approximation scheme}\label{sec:approximation}

The idea of meshfree finite difference methods is that at each node $x_i\in\G$ in the point cloud, we examine all other nodes within the search neighbourhood $B(x_i,r)\cap\G$.  An appropriate subset $\{x_j; j\in\Nf(i)\}$ of these points is then selected to form the local finite difference stencil.  Finally, these are used to construct an approximation of the form
\[ \frac{\partial^2u}{\partial e_\theta^2} \approx \tilde{F}_\theta(x_i,u_i,u_i-u_j; j\in\Nf(i)). \]

Meshfree methods have previously been used to approximate several PDE operators~\cite{Belytschko,Demkowicz,DuarteOden,Iliev,LaiZhao,LiszkaFDM,Liszka,Seibold_minimal}.  However, the approaches contained in these works do not apply to the construction of monotone approximations of fully nonlinear or degenerate operators.  In fact, from the results of~\cite{Kocan,MotzkinWasow}, we expect that in general no finite search neighbourhood will be sufficient for the construction of an approximation that is both monotone and consistent (in the sense that the formal discretisation error goes to zero as $h\to0$).

We propose monotone approximations that are not consistent in the traditional sense; in particular, they will not give exact results on quadratic functions.  Instead, we will accept an additional source of discretisation error $d\theta$ relating to how well the stencil is aligned with the direction $e_\theta$. To build a convergent method in this framework, we will allow the search radius $r$ to depend on the spatial resolution $h$, with the total number of points in the search neighbourhood approaching infinity as the point cloud is refined.

Consider any interior point $x_0\in\G\cap\Omega$.  Each point $y_j,\,j=1,\ldots,N$ in the search neighbourhood $B(x_0,r)\cap\G$  can be expressed in polar coordinates $(h_j,\theta_j)$ in terms of the rotated coordinate frame defined by the vectors $x_0+e_\theta$ and $x_0+e_{\theta+\pi/2}$:
\bq\label{eq:polar}
y_j = \begin{cases}
(h_j,d\theta_j), & 0 \leq d\theta_j < \pi/2\\
(h_j,\pi-d\theta_j), & \pi/2 \leq \pi-d\theta_j < \pi\\
(h_j,\pi+d\theta_j), & \pi \leq \pi+d\theta_j < 3\pi/2\\
(h_j,2\pi-d\theta_j), & 3\pi/2 \leq 2\pi-d\theta_j < 2\pi.
\end{cases}
\eq
Note that $d\theta_j$ measures the angular distance between the point $x_j$ and the given direction vector $x_0+e_\theta$; in each case, $d\theta_j\in[0,\pi/2]$.  In particular, we assume that if the search radius is large enough, each $d\theta_j$ will be less than some pre-specified angular resolution $d\theta$.  Conditions needed to ensure the existence of these discretisation points will be established in subsection~\ref{sec:existence}.

In each quadrant, we select the point that aligns most closely with the given direction vector $x_0 + e_\theta$,
\bq\label{eq:stencil} x_i = \argmin\limits_{y_j\in B(x_0,r)\cap\G}\left\{d\theta_j\mid (i-1)\pi/2 \leq \theta_j < i\pi/2\right\}, \quad i = 1, \ldots, 4. \eq
The existence of these points is established in subsection~\ref{sec:existence}.
If more than one value $y_j$ yields the same angular distance, we select the value with the smallest radial coordinate $h_j$.  See Figure~\ref{fig:stencil} for an illustration of the resulting finite difference stencil.

\begin{figure}[htp]
\centering
\subfigure[]{
\includegraphics[width=0.52\textwidth]{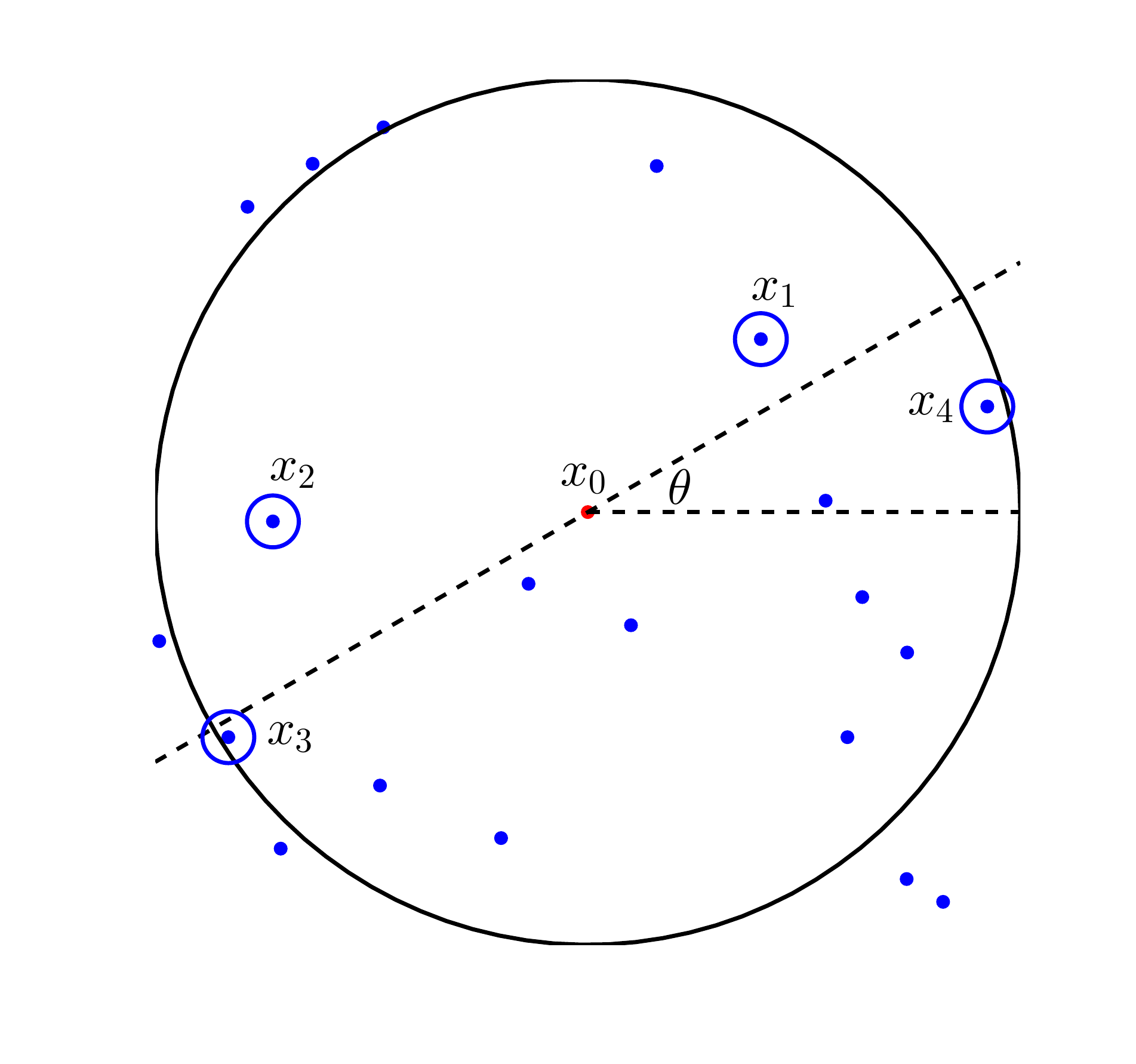}\label{fig:stencil1}}
\subfigure[]{
\includegraphics[width=0.65\textwidth]{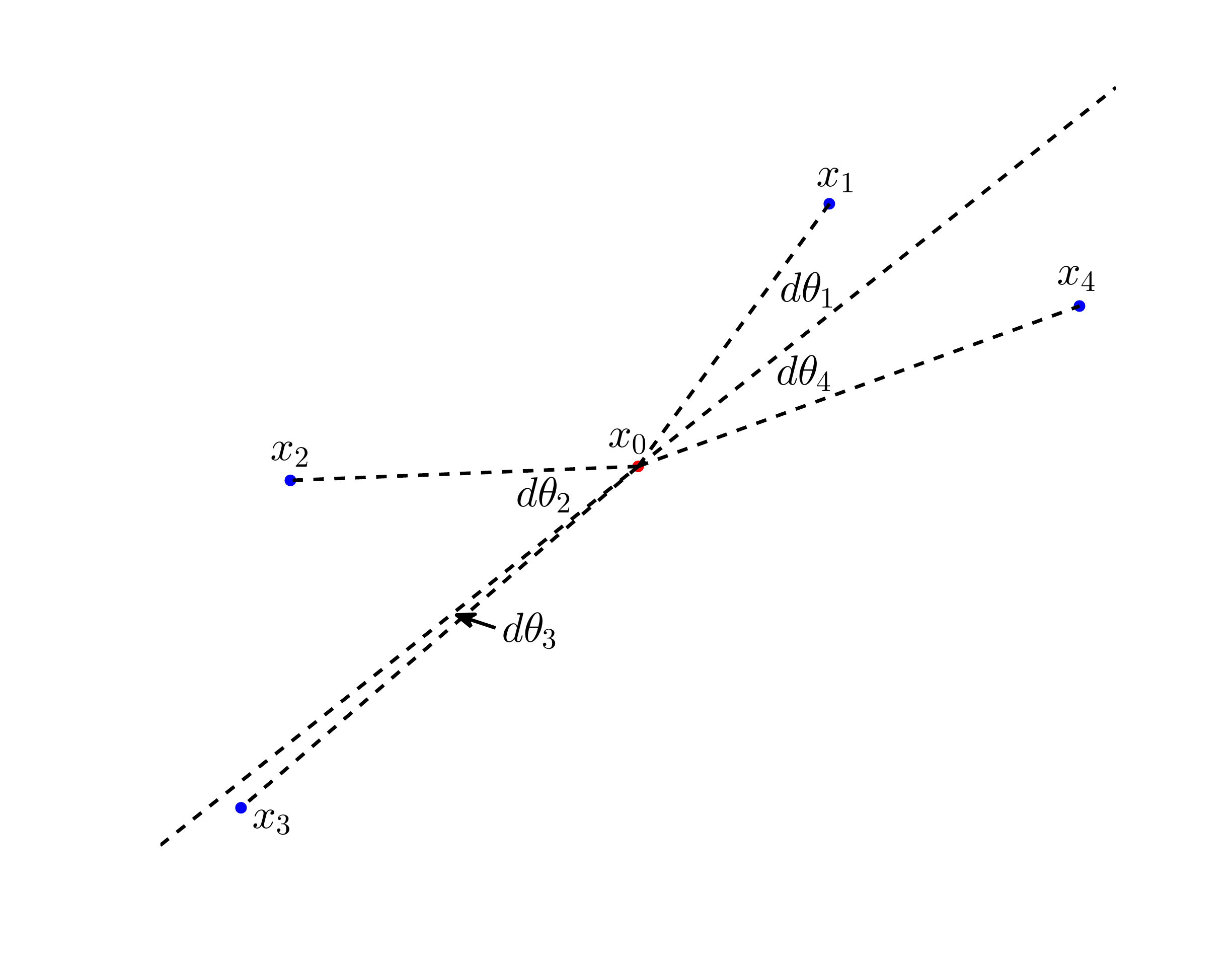}\label{fig:stencil2}}
\caption{A finite difference stencil chosen from a point cloud.}
\label{fig:stencil}
\end{figure}

Next, we seek to approximate the PDE using this stencil.  We look for an approximation of the form
\begin{align*}
\frac{\partial^2u}{\partial e_\theta^2} &\approx \sum\limits_{i=1}^4 a_i(u(x_i)-u(x_0))\\
  &= \sum\limits_{i=1}^4 a_i\left[h_i\cos\theta_iu_{\theta}(x_0) + h_i\sin\theta_iu_{{\theta+\pi/2}}(x_0)+\frac{1}{2}h_i^2\cos^2\theta_iu_{\theta\theta}(x_0)\right.\\&\phantom{=}\left.+\bO(h_i^3+h_i^2\sin d\theta_i)\right].
\end{align*}

Consistency and (negative) monotonicity require at a minimum
\bq\label{eq:conditions}
\begin{cases}
\sum\limits_{i=1}^4 a_ih_i\cos\theta_i = 0\\
\sum\limits_{i=1}^4 a_ih_i\sin\theta_i = 0\\
\sum\limits_{i=1}^4 \frac{1}{2}a_ih_i^2\cos^2\theta_i = 1\\
a_i \geq 0.
\end{cases}
\eq


This type of approximation is unusual in that it does not take into account the other second derivatives $u_{{\theta+\pi/2},{\theta+\pi/2}}$ and $u_{\theta ,{\theta+\pi/2}}$.  Because of this, the resulting approximation scheme need not be exact on quadratic functions.  However, as long as the values of $d\theta_i$ are small (i.e. the points are well aligned with the direction~$e_\theta$), the contribution from these second derivatives is also expected to be small.

Ignoring the condition $a_i \geq 0$, the consistency conditions~\eqref{eq:conditions} lead to a system of three linear equations in four unknowns.  Existence of a positive solutions is guaranteed, as we demonstrate below, and in general we can expect infinitely many positive solutions.  One way to select a particular solution is to augment the system with an additional symmetry condition.  A natural choice is
\bq\label{eq:symmetry}
a_1h_1\sin\theta_1 + a_4h_4\sin\theta_4 = 0.
\eq
Among other things, this ensures that if one of the neighbours (say $x_1$) exactly aligns with the $e_\theta$ direction so that $\sin\theta_1 = 0$, the non-aligned neighbour $x_4$ will receive no weight in the approximation scheme.  We also observe that this condition combined with~\eqref{eq:conditions} ensures a similar condition for the neighbours $x_2, x_3$ that approximately align with the $-e_\theta$ direction:
\[ a_2h_2\sin\theta_2 + a_3h_3\sin\theta_3 = 0. \]

We can now explicitly solve the linear system~\eqref{eq:conditions}-\eqref{eq:symmetry}.  To make the description more compact, we introduce the notation
\bq\label{eq:trig} C_i = h_i\cos\theta_i = \bO(h_i), \quad S_i = h_i\sin\theta_i = \bO(h_i d\theta_i).\eq
Then a solution of~\eqref{eq:conditions}-\eqref{eq:symmetry} is 
\bq\label{eq:coeffs}
\begin{split}
a_1 &= \frac{2S_4(C_3S_2-C_2S_3)}{(C_3S_2-C_2S_3)(C_1^2S_4-C_4^2S_1)-(C_1S_4-C_4S_1)(C_3^2S_2-C_2^2S_3)}\\
a_2 &= \frac{2S_3(C_1S_4-C_4S_1)}{(C_3S_2-C_2S_3)(C_1^2S_4-C_4^2S_1)-(C_1S_4-C_4S_1)(C_3^2S_2-C_2^2S_3)}\\
a_3 &= \frac{-2S_2(C_1S_4-C_4S_1)}{(C_3S_2-C_2S_3)(C_1^2S_4-C_4^2S_1)-(C_1S_4-C_4S_1)(C_3^2S_2-C_2^2S_3)}\\
a_4 &= \frac{-2S_1(C_3S_2-C_2S_3)}{(C_3S_2-C_2S_3)(C_1^2S_4-C_4^2S_1)-(C_1S_4-C_4S_1)(C_3^2S_2-C_2^2S_3)}.
\end{split}
\eq

We note that because $x_i$ lies in the $ith$ quadrant, we have
\[ C_1, C_4, S_1, S_2 \geq 0, \quad C_2, C_3, S_3, S_4 \leq 0. \]
This ensures that the coefficients~\eqref{eq:coeffs} satisfy the positivity condition $a_i \geq 0$.

We can easily verify that each coefficient $a_i$ has a size on the order of at most
\[ a_i = \bO\left(\frac{1}{h_i^2}\right). \]
For example,
\begin{align*}
{a_1} &= \frac{2S_4(C_3S_2-C_2S_3)}{(C_3S_2-C_2S_3)(C_1^2S_4-C_4^2S_1)-(C_1S_4-C_4S_1)(C_3^2S_2-C_2^2S_3)}\\
  &\leq \frac{2S_4(C_3S_2-C_2S_3)}{(C_3S_2-C_2S_3)(C_1^2S_4-C_4^2S_1)}= \frac{2S_4}{C_1^2S_4-C_4^2S_1}\\
	&\leq \frac{2S_4}{C_1^2S_4}=\bO\left(\frac{1}{h_1^2}\right).
\end{align*}

By construction, the overall spatial and angular resolution of the scheme satisfy
\[
\max\{h_i\} \leq r, \quad \max\{d\theta_i\} \leq d\theta.
\]
Thus the resulting (negative) monotone approximation scheme has the form
\bq\label{eq:fd}
\Dt_{\theta\theta}u(x_0) \equiv \sum\limits_{i=1}^4{a_i(u(x_i)-u(x_0))} =\frac{\partial^2u(x_0)}{\partial\theta^2}+ \bO(r + d\theta).
\eq

We remark that in the special case of a Cartesian grid and a direction $e_\theta$ that aligns with the grid, the approximation resulting from the coefficients~\eqref{eq:coeffs} reduces to the usual centred difference discretisation.

\subsection{Existence of consistent, monotone scheme}\label{sec:existence}
Next, we establish conditions on the point cloud that will ensure the existence of a monotone and consistent scheme.

In order to construct the finite difference approximation, it was necessary to posit the existence of a node $x_i$ in each quadrant such that $0 \leq d\theta_i \leq d\theta < \pi/2$ and $\abs{x_i-x_0} \leq r$.  In this section, we describe conditions on the point cloud $\G$ and the associated search radius $r$ that guarantee this is true.

We consider two different scenarios: points that are a distance of at least $r$ from the boundary $\partial\Omega$ and points that are within a distance $r$ of the boundary.

First we consider points sufficiently far from the boundary. 
\begin{lemma}[Existence of Stencil (Interior)]\label{lem:existInt}
Choose any $x_0\in\Omega$ such that $\text{dist}(x,\partial\Omega) \geq r$.  Then the four discretisation points $x_i\in\G$ defined by~\eqref{eq:stencil} exist.
\end{lemma}

\begin{proof}
We demonstrate the existence of a suitable node $x_1$ in the first quadrant; the other cases are analogous.

Our goal is to show that the set 
\bq\label{eq:wedge}
\G\cap\left\{x_0+te_\phi\mid\phi\in[\theta,\theta+d\theta], t\in(0,r]\right\}
\eq
is non-empty.  
Consider the closed ball $\bar{B}(x_0+(r-h)e_{\theta+d\theta/2},h)$, recalling that
\[ r = h(1+\sin(d\theta/2)+\cos(d\theta/2)\cot(d\theta/2)). \] 
Using elementary geometric arguments, we see that this small ball is contained within the above wedge. See Figure~\ref{fig:existInt}.  

From the definition of the spatial resolution $h$, any ball of radius~$h$ must contain a node $x_1\in\G$.  Since this small ball is contained within the wedge~\eqref{eq:wedge}, we have successfully identified a node~$x_1$ within the wedge.  Thus the wedge contains an appropriate discretisation node and the monotone stencil exists.
\qed\end{proof}
\begin{figure}
\centering
\includegraphics[width=\textwidth]{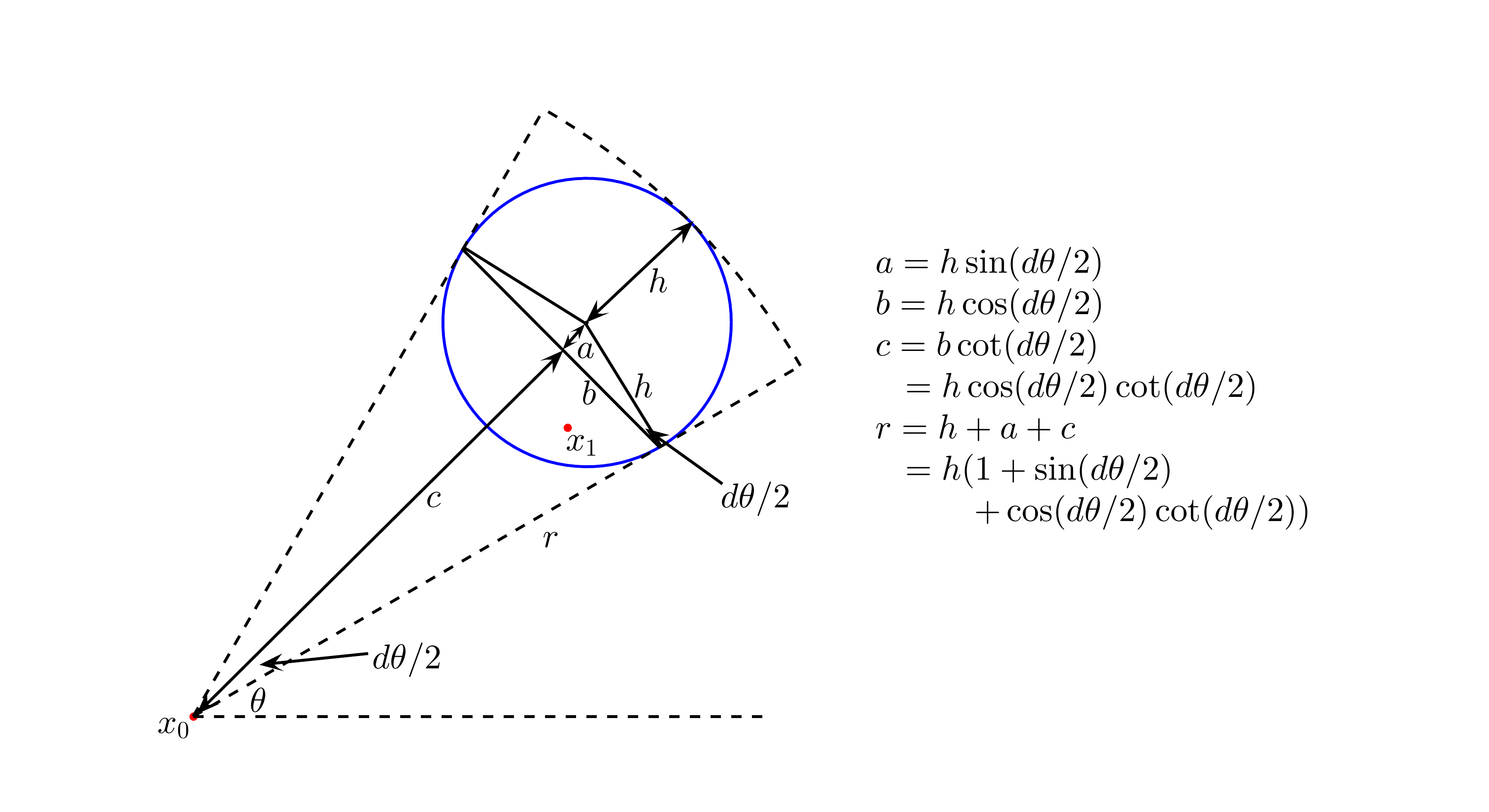}
\caption{A node $x_1$ exists within the given search neighbourhood.}
\label{fig:existInt}
\end{figure}

Secondly, we consider points close to the boundary of the domain.  This is a more delicate calculation since a ball of the usual search radius~$r$ may not be contained in the domain, and the argument used to prove Lemma~\ref{lem:existInt} breaks down.  
Indeed, for certain degenerate PDEs, more traditional methods posed on a uniform grid~\cite{Bonnans_HJB,ObermanWS} will necessarily be either inconsistent or non-monotone at points near the boundary where the full stencil width cannot be accessed~\cite{Kocan}.

 In order to ensure the existence of appropriate neighbours close to the boundary, we need to require that the boundary of the domain is more highly resolved than the interior.  This boundary resolution is characterised by the parameter $h_B$, which will typically be less than the overall resolution~$h$.

\begin{lemma}[Existence of Stencil (Boundary)]\label{lem:existBdy}
Choose any $x_0\in\Omega$ such that $\text{dist}(x,\partial\Omega) < r$.
If the boundary resolution of the point cloud $\G$ satisfies $h_B \leq 2\delta\tan(d\theta/2)$ and $d\theta$ is sufficiently small (depending on the regularity of the domain) then the four discretisation points $x_i\in\G$ defined by~\eqref{eq:stencil} exist.
\end{lemma}
\begin{proof}
Our goal is to show that the set 
\bq\label{eq:wedgeBound}
\G\cap\bar{\Omega}\cap\left\{x_0+te_\phi\mid\phi\in[\theta,\theta+d\theta], t\in(0,r]\right\}
\eq
is non-empty.  If the wedge is contained entirely within the domain $\Omega$, the proof proceeds as with Lemma~\ref{lem:existInt}.  

Suppose instead, that this wedge intersects $\partial\Omega$.  In particular,  we let $y$, $z$ be the first points of intersection of the rays $x_0+te_\theta$, $x_0+te_{\theta+d\theta}$ with the boundary $\partial\Omega$.  For small enough $d\theta$, the arc of the boundary between $y$ and $z$ lies completely inside the search neighbourhood,
\[ \abs{w-x_0} < r \text{ whenever } w\in\partial\Omega \text{ lies between }y \text{ and }z. \]

By definition, $\text{dist}(x_0,\partial\Omega) \geq \delta$.  Thus the arclength of the boundary contained in the search neighbourhood is at least $2\delta\tan(d\theta/2) \geq h_B$; see Figure~\ref{fig:existBdy}.  By the definition of the boundary resolution, this must contain a discretisation node.
\qed\end{proof}
\begin{figure}
\includegraphics[width=0.75\textwidth]{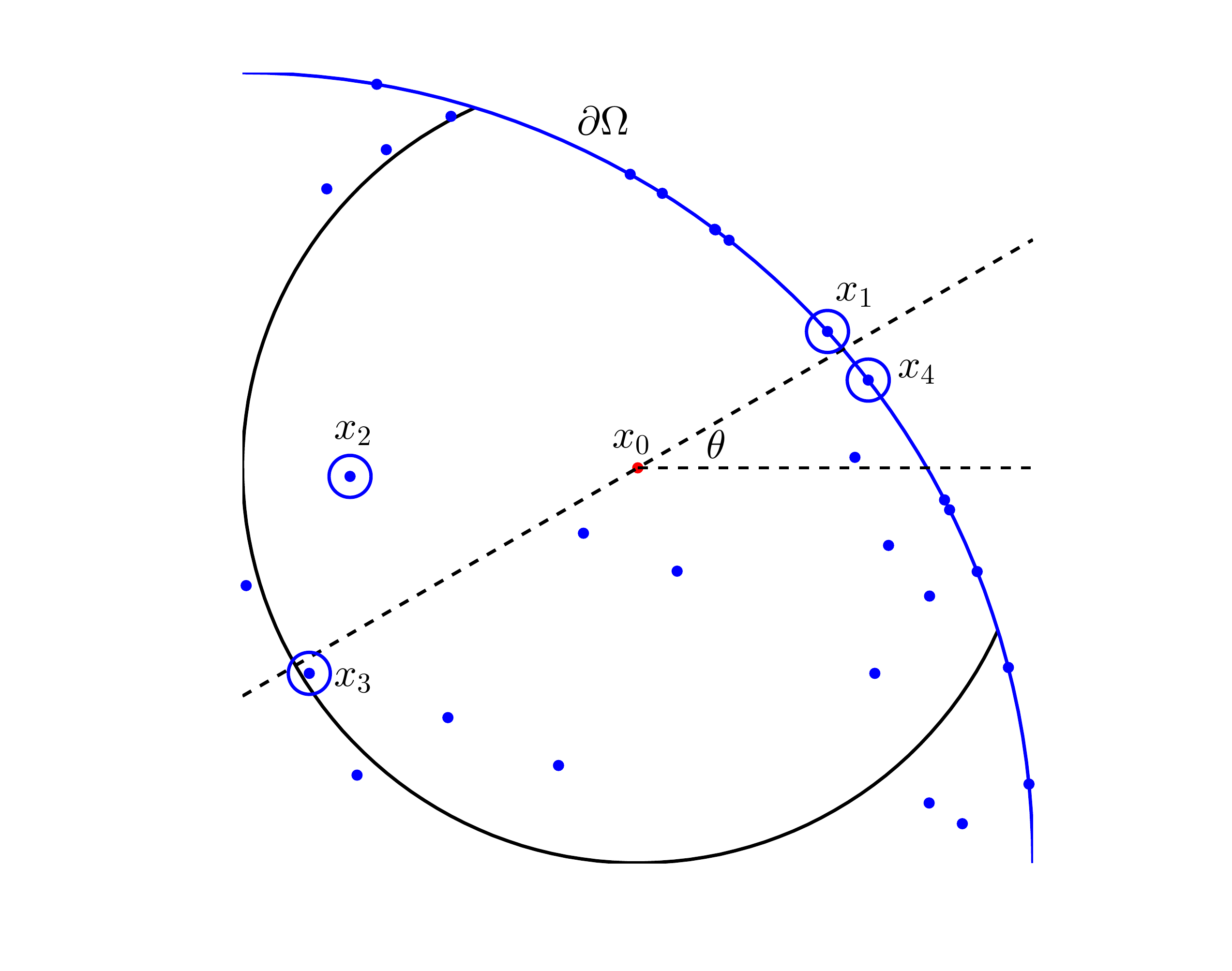}
\includegraphics[width=0.75\textwidth]{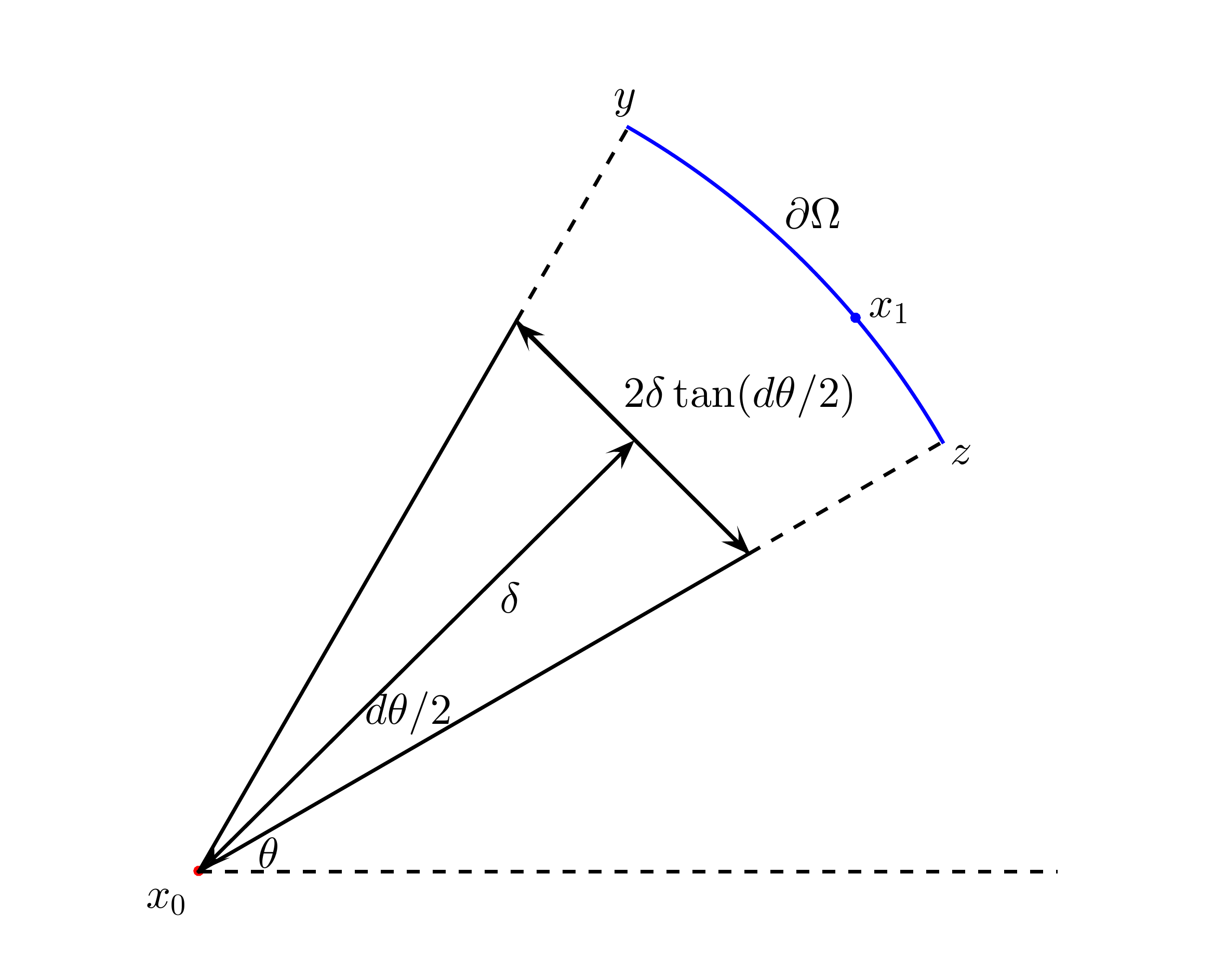}
\caption{A node $x_1$ exists within the given search neighbourhood near the boundary.}
\label{fig:existBdy}
\end{figure}

These two results immediately yield the existence of a monotone discretisation.
\begin{theorem}[Existence of Stencil]\label{thm:exist}
Let $\G$ be a point cloud with boundary resolution $h_B \leq 2\delta\tan(d\theta/2)$ and let $x_0\in\Omega$.  If the angular resolution $d\theta$ is sufficiently small then the four discretisation points $x_i\in\G$ defined by~\eqref{eq:stencil} exist.
\end{theorem}

It is also important to verify that the scheme is consistent; that is, the discretisation error should go to zero as the point cloud is refined.    We recall that the search radius is given by
\[ r = h(1+\sin(d\theta/2)+\cos(d\theta/2)\cot(d\theta/2)) = \bO\left(\frac{h}{d\theta}\right).\]
Then the overall discretisation error in~\eqref{eq:fd} is
\[ \bO\left(r+d\theta\right) = \bO\left(\frac{h}{d\theta}+d\theta\right). \]
Consistency requires that as $h\to0$, both $h/d\theta\to0$ and $d\theta\to0$.  In particular, an optimal choice is $d\theta = \bO(\sqrt{h})$.  This requires a search radius of size $r = \bO(\sqrt{h})$ and yields a formal discretisation error of $\bO(\sqrt{h})$.

We emphasise also that the boundary needs to be sufficiently well resolved in order to construct these schemes.  For our optimal choice, we suppose that the distance~$\delta$ between any interior node and the boundary is on the order of $\bO(h)$.  Then the spatial resolution of the boundary should be
\[ h_B \leq 2\delta\tan(d\theta/2) = \bO(h^{3/2}). \]
The more traditional alternative to this higher boundary resolution is to use a lower-order or even inconsistent scheme near the boundary, which leads to a computed solution containing a boundary layer~\cite{Bonnans_HJB,ObermanWS}.    

\subsection{Convergent approximation of nonlinear equations}\label{sec:convergence}

Now we demonstrate that we can use meshfree finite differences to construct convergent approximations of fully nonlinear elliptic PDEs of the form~\eqref{eq:PDE1} and~\eqref{eq:PDE2}. 

In the case of~\eqref{eq:PDE1}, which involves second directional derivatives in finitely many directions, we use the approximation
\bq\label{eq:approx1}
\tilde{F}_i[u]  \equiv F(x_i,u_i,\Dt_{\theta\theta}u_i;\theta\in\Af) = 0.
\eq

As in~\cite{ObermanWS}, the eigenvalues of the Hessian in 2D can be characterised (via Rayleigh-Ritz) as the minimal and maximal second directional derivatives,
\[ \lambda_-(D^2u) = \min\limits_{\theta\in[0,2\pi)}\frac{\partial^2u}{\partial e_\theta^2}, \quad  \lambda_+(D^2u) = \max\limits_{\theta\in[0,2\pi)}\frac{\partial^2u}{\partial e_\theta^2}.\]
We approximate these by computing the minima (maxima) over a finite subset
\bq\label{eq:directions}
\tilde{\Af} = \left\{jd\theta\mid j = 0, \ldots, \floor{\frac{2\pi}{d\theta}}\right\},
\eq
which introduces the directional resolution error into a second part of the approximation.  Then we can approximate~\eqref{eq:PDE2} by
\bq\label{eq:approx2}
\tilde{F}_i[u] \equiv F\left(x_i,u_i,\min\limits_{\theta\in\tilde{\Af}}\Dt_{\theta\theta}u_i,\max\limits_{\theta\in\tilde{\Af}}\Dt_{\theta\theta}u_i\right) = 0.
\eq

We note that for boundary nodes $x_i\in\G\cap\partial\Omega$, we simply enforce the Dirichlet boundary data and the monotone scheme is
\[ u(x_i)-g(x_i) = 0. \]

These schemes are consistent and monotone.
\begin{lemma}[Consistency]\label{lem:consistent1}
Let $F$ be a continuous, degenerate elliptic operator.
Then the scheme~\eqref{eq:approx1} is a consistent approximation of~\eqref{eq:PDE1}.
\end{lemma}
\begin{proof}
Let $u\in C^2$.  Then 
\begin{align*}
\tilde{F}_i[u] &= F(x_i,u_i,\Dt_{\theta\theta}u_i;\theta\in\Af)\\
  &= F\left(x_i,u_i,\frac{\partial^2u_i}{\partial e_{\theta}^2}+\bO(r+d\theta);\theta\in\Af\right)\\
	&= F\left(x_i,u_i,\frac{\partial^2u_i}{\partial e_{\theta}^2};\theta\in\Af\right) + \bO(\kappa(r+d\theta))
\end{align*}
where $\kappa$ is the modulus of continuity of $F$.
\qed\end{proof}

\begin{lemma}[Consistency]\label{lem:consistent2}
Let $F$ be a continuous, degenerate elliptic operator.
Then the scheme~\eqref{eq:approx2} is a consistent approximation of~\eqref{eq:PDE2}.
\end{lemma}
\begin{proof}
Let $u\in C^2$.  From~\cite[Lemma 5.3]{ObermanWS}, 
\[ \min\limits_{\theta\in\tilde{\Af}}\frac{\partial^2u}{\partial e_\theta^2} = \lambda_-(D^2u) + \bO(d\theta^2), \quad \max\limits_{\theta\in\tilde{\Af}}\frac{\partial^2u}{\partial e_\theta^2} = \lambda_+(D^2u) + \bO(d\theta^2). \]
Then 
\begin{align*}
\tilde{F}_i[u] &= F\left(x_i,u_i,\min\limits_{\theta\in\tilde{\Af}}\Dt_{\theta\theta}u_i,\max\limits_{\theta\in\tilde{\Af}}\Dt_{\theta\theta}u_i\right)\\
  &= F\left(x_i,u_i,\min\limits_{\theta\in\tilde{\Af}}\frac{\partial^2u_i}{\partial e_\theta^2}+\bO(r+d\theta),\max\limits_{\theta\in\tilde{\Af}}\frac{\partial^2u_i}{\partial e_\theta^2}+\bO(r+d\theta)\right)\\
	&= F\left(x_i,u_i,\lambda_-(D^2u_i)+\bO(r+d\theta),\lambda_+(D^2u_i)+\bO(r+d\theta)\right)\\
	&= F\left(x_i,u_i,\lambda_-(D^2u_i),\lambda_+(D^2u_i)\right) + \bO(\kappa(r+d\theta))
\end{align*}
where $\kappa$ is the modulus of continuity of $F$.
\qed\end{proof}

\begin{remark}
If the nonlinear operator $F$ is Lipschitz continuous and we make use of the optimal scaling $r, d\theta = \bO(\sqrt{h})$, the formal consistency error of the scheme is $\bO(\sqrt{h})$.
\end{remark}

\begin{lemma}[Monotonicity]\label{lem:monotone}
Let $F$ be a continuous, degenerate elliptic operator.  Then the approximations~\eqref{eq:approx1} and~\eqref{eq:approx2} are monotone for sufficiently small~$d\theta$.
\end{lemma}
\begin{proof}
The schemes at a point $x_i$ rely on approximations of the second directional derivatives $u_{\theta\theta}$ that are non-increasing in each $u_i-u_j$; these can be constructed by Theorem~\ref{thm:exist}.
The functions~$F$ are non-decreasing in the argument $u_i$ and non-increasing in the $u_{\theta\theta}$.  Thus they are non-decreasing in each $u_i-u_j$ and therefore monotone.
\qed\end{proof}

\begin{theorem}[Convergence]\label{thm:converge}
Let $F$ be a continuous, degenerate elliptic operator with a comparison principle and let $u$ be the unique viscosity solution of the PDE~\eqref{eq:PDE1} (or~\eqref{eq:PDE2}).  Consider a sequence of point clouds $\G^n$, with parameters defined as in subsection~\ref{sec:notation}, which satisfy the following conditions.
\begin{itemize}
\item The spatial resolution $h^n\to0$ as $n\to\infty$.
\item The desired angular resolution $d\theta^n$ is chosen so that both $h^n/d\theta^n\to0$ and $d\theta^n\to0$ as $h^n\to0$.
\item The boundary resolution $h_B^n \leq 2\delta^n\tan(d\theta^n/2)$.
\end{itemize}
Let $u^n$ be the solution of the approximation scheme~\eqref{eq:approx1} (or~\eqref{eq:approx2}).  Then as $n\to\infty$, $u^n$ converges uniformly to $u$.
\end{theorem}

\begin{proof}
By Lemmas~\ref{lem:consistent1}-\ref{lem:monotone} the schemes are consistent and monotone.  Therefore they converge to the viscosity solution of the underlying PDE~\cite{BSnum,ObermanDiffSchemes}.
\qed\end{proof}

\subsection{Filtered schemes}\label{sec:filter}

One of the apparent drawbacks of the meshfree approximation scheme described above is its low accuracy---formally, it is at best $\bO(\sqrt{h})$.  Indeed, for certain non-degenerate equations, schemes as accurate as $\bO(h^2)$ may be possible using regular grids, at least in the absence of boundary effects~\cite{Bonnans_HJB}.  One clear advantage of the meshfree schemes is their ability to preserve consistency and order of accuracy near boundary points and in complicated domains.  More importantly, though, these monotone schemes can provide the foundation for higher-order convergent filtered schemes as in~\cite{FOFiltered}.  This opens up many possibilities (finite difference, finite element, etc.) for designing higher-order, provably convergent schemes on general meshes or point clouds.

To accomplish this, we let $F_N[u]$ be any higher-order scheme, which need not be monotone or even stable, and may be defined on a very general mesh.  Using the approach presented in this article, we can construct a monotone approximation scheme $F_M[u]$ that is defined on the same mesh (or point cloud).  These can be combined into the filtered scheme
\bq\label{eq:fdfilter}
F_F[u] = F_M[u] + \epsilon(h)S\left(\frac{F_A^h-F_M^h}{\epsilon(h)}\right)
\eq
where the filter $S$ is given by
\bq\label{eq:filter}
S(x) = \begin{cases}
x & \abs{x} \leq 1 \\
0 & \abs{x} \ge 2\\
-x+ 2  & 1\le x \le 2 \\
-x-2  & -2\le x\le -1.
\end{cases} 
\eq

As long as $\epsilon(h)\to0$ as $h\to0$, this approximation converges to the viscosity solution of the PDE.  Moreover, if $\epsilon(h)$ is larger than the discretisation error of the monotone scheme, the formal accuracy of the filtered scheme is the same as the formal accuracy of the non-monotone scheme.

\section{Computational Examples}\label{sec:compute}

In this section, we provide several computational examples to demonstrate the correctness and flexibility of our meshfree finite difference approximations.  In each example, $N$ denotes the total number of discretisation points, which includes interior and boundary points.  Unless otherwise stated, we choose $d\theta = 2\sqrt{h}$ in each example.

\subsection{Linear degenerate equation}\label{sec:deg}

For our first example, we consider the linear degenerate equation
\bq\label{eq:lin}
\begin{cases}
-u_{\nu\nu}(x,y) = 0, & x^2+y^2 < 1\\
u(x,y) = \sin(2\pi(x-\sqrt{8}y)), & x^2+y^2 = 1
\end{cases}
\eq
where $\nu=(\sqrt{8},1)$.  The exact solution is 
\[ u(x,y) = \sin(2\pi(x-\sqrt{8}y)), \]
which is plotted in Figure~\ref{fig:DegSol}.

We note that this is an example of an operator for which no monotone, consistent approximation can be constructed on a finite stencil on a Cartesian grid~\cite{MotzkinWasow}.

We first solve this equation using a point cloud generated by a uniform Cartesian mesh restricted to the interior of the unit circle, which is augmented by $\bO(h^{-3/2})$ points uniformly distributed on the boundary of the circle; see Figure~\ref{fig:DegMesh1}.  The discretised problem is a sparse, diagonally dominant linear system, which we solve using Matlab backslash.

Next, we demonstrate that our approximations converge even on highly unstructured point clouds.  To do this, we use randomly selected points in the interior of the unit circle, augmented by additional points randomly distributed on the boundary; see Figure~\ref{fig:DegMesh2}.  The point cloud is refined by randomly adding additional points.  

We note that if interior points are located too close to the boundary, the parameter $\delta$ can become extremely small, and it may not be possible to satisfy the condition on the boundary resolution $h_B$ given in Theorem~\ref{thm:converge}.  To overcome this challenge (for both uniform and random point clouds), we simply remove points if no monotone stencil can be found within the given search radius $r$.  This has the effect of removing points that are too close to the boundary of the domain (thus increasing $\delta$), and explains why the total number of points $N$ are slightly different in the two examples.

Convergence results for both tests are presented together in Table~\ref{table:deg}.  In the case of the Cartesian mesh, we observe convergence even though the approximation is not consistent on any fixed stencil.  However, because the search radius is  large enough to ensure decreasing angular resolution error, the predicted convergence is observed.  Although the random point cloud is highly unstructured, we again observe convergence in this setting, with a rate that is nearly unchanged.  

\begin{figure}
\centering
{\subfigure[]{\includegraphics[width=0.3\textwidth]{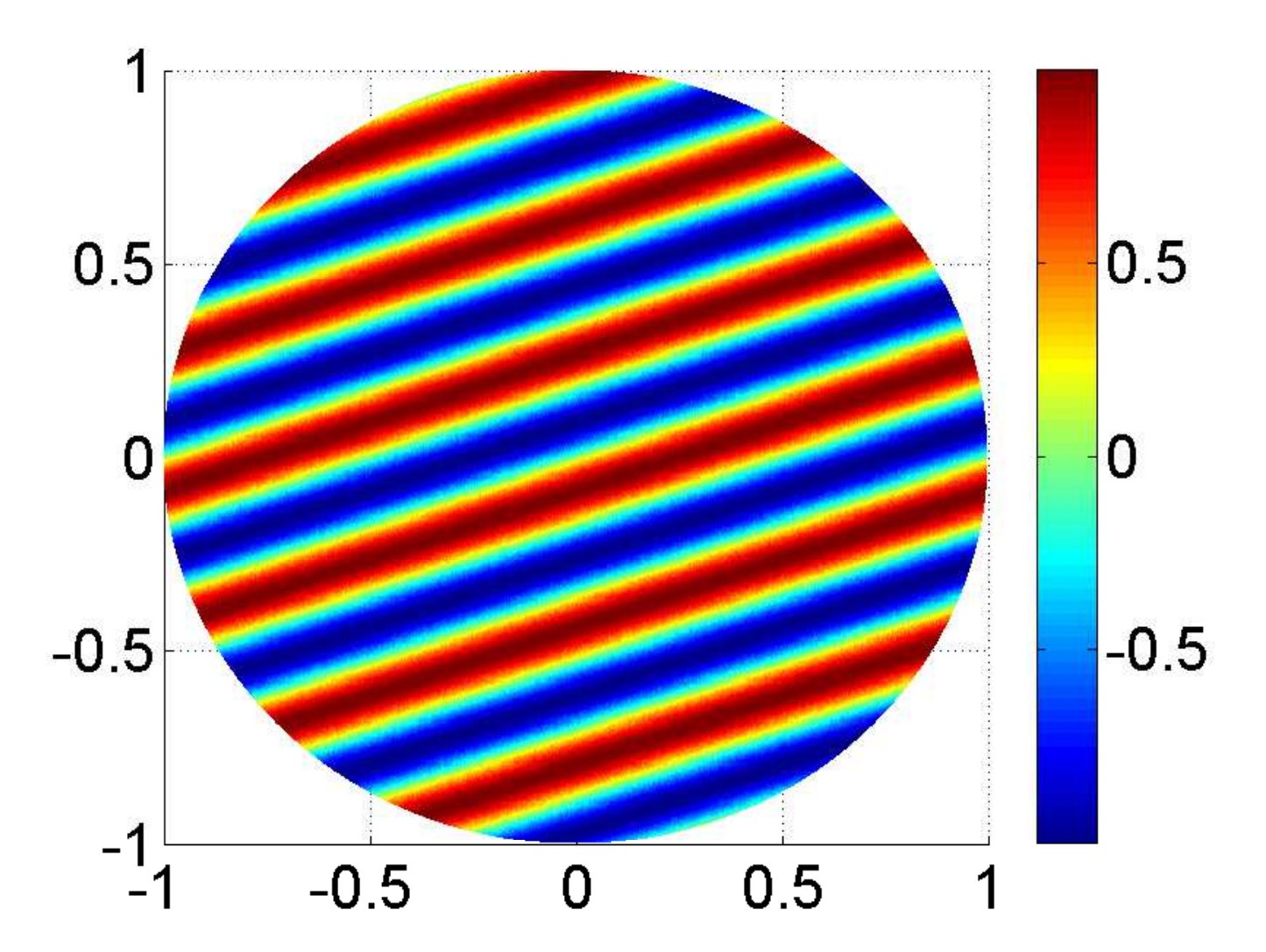}\label{fig:DegSol}}}
{\subfigure[]{\includegraphics[width=0.3\textwidth]{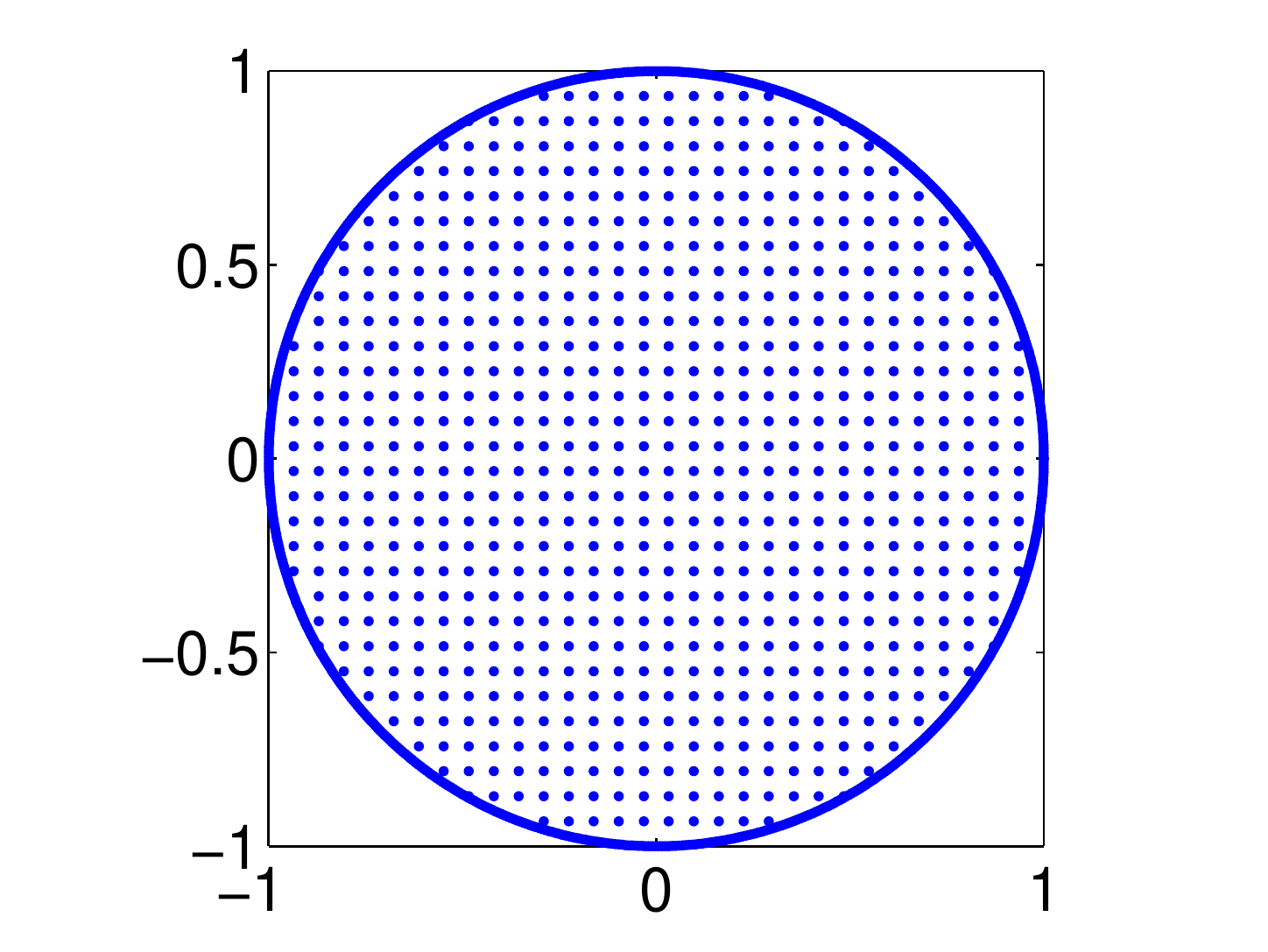}\label{fig:DegMesh1}}}
{\subfigure[]{\includegraphics[width=0.3\textwidth]{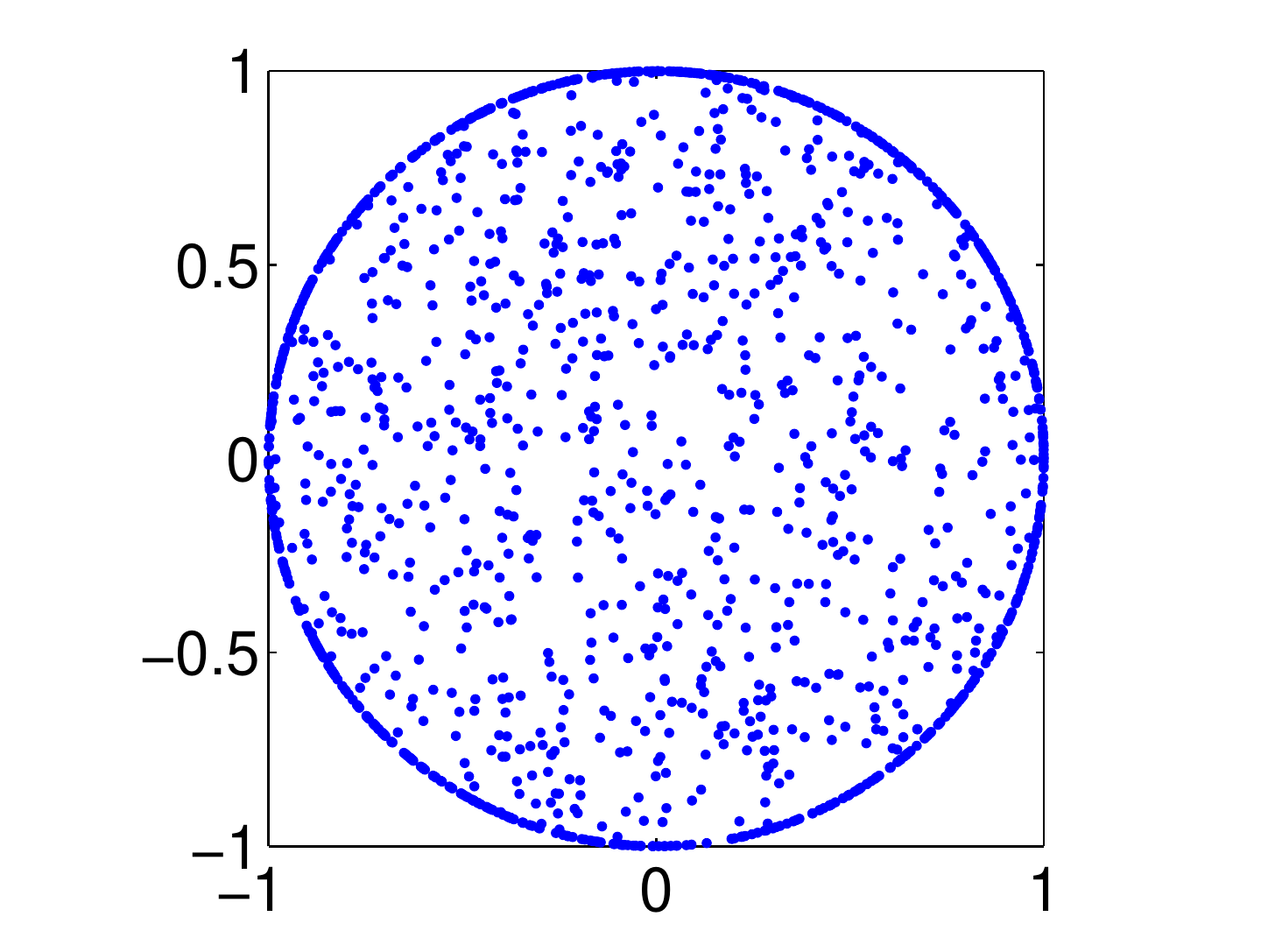}\label{fig:DegMesh2}}}
\caption{\subref{fig:DegSol}~Solution of the linear degenerate equation~\eqref{eq:lin}, \subref{fig:DegMesh1}~uniform point cloud, and \subref{fig:DegMesh2}~random point cloud.}
\label{fig:deg}
\end{figure}

\begin{table}[htp]
\centering
\small
\begin{tabular}{cccc|ccc}
\multicolumn{4}{c|}{Uniform Point Cloud} & \multicolumn{3}{c}{Random Point Cloud}\\
$h$ & $N$ & Max Error  & Rate ($N$) & $N$ & Max Error & Rate ($N$)\\
\hline
2/32  & 1,462  & $3.0\times10^{-1}$ & ---&1,459   &$7.5\times10^{-1}$&---\\
2/64  & 5,143  & $1.3\times10^{-1}$ & 0.7&5,138   &$2.6\times10^{-1}$&0.9\\
2/128 & 18,435 & $5.8\times10^{-2}$ & 0.6&18,430  &$1.1\times10^{-1}$&0.7\\
2/256 & 67,423 & $2.2\times10^{-2}$ & 0.7&67,412  &$2.6\times10^{-2}$&1.1\\
2/512 & 251,349& $6.1\times10^{-3}$ & 1.0&251,335 &$7.2\times10^{-3}$&1.0
\end{tabular}
\caption{Convergence results for the linear degenerate equation~\eqref{eq:lin}.}
\label{table:deg}
\end{table}

\subsection{Convex envelope}\label{sec:CE}
In our second example, we demonstrate the convergence of a meshfree finite difference approximation of the fully nonlinear convex envelope equation
\bq\label{eq:CE}
\begin{cases}
\max\{-\lambda_-(D^2u), u-g\} = 0, &x \in\Omega\\
u = 0.5, & x\in\partial\Omega.
\end{cases}
\eq
The equation is posed on an ellipse with semi-major axis equal to one and semi-minor axis equal to one-half, which is rotated through an angle of $\phi=\pi/6$.  The obstacle $g$ consists of two cones,
\begin{align*}
g_1(x,y) &= \sqrt{(x\cos\phi+y\sin\phi+0.5)^2+(-x\sin\phi + y\cos\phi)^2}\\
g_2(x,y) &= \sqrt{(x\cos\phi+y\sin\phi-0.5)^2+(-x\sin\phi + y\cos\phi)^2}\\
g(x,y) &= \min\left\{g_1(x,y),g_2(x,y),0.5\right\}
\end{align*}
and the exact solution is
\[
u(x,y) = \begin{cases}
\min\{g_1(x,y),g_2(x,y)\}, & \abs{x\cos\phi+y\sin\phi} \geq 0.5\\
\abs{-x\sin\phi + y\cos\phi}, & \abs{x\cos\phi+y\sin\phi} < 0.5.
\end{cases}
\]
See Figure~\ref{fig:CE}.  We note that this solution is only Lipschitz continuous, and the equation must be understood in a weak sense.

We perform computations using a uniform point cloud augmented by a uniform discretisation of the boundary. 
The discrete system is solved using a policy iteration procedure.  To do this, we note that the PDE~\eqref{eq:CE} (and its discretisation) can be written in the form
\[ \max\limits_\alpha\{L^\alpha u - g^\alpha\} = 0 \]
where the $L^\alpha$ are diagonally dominant linear operators---either the identity or second directional derivatives.  Then we can use the update scheme
\begin{align*}
\alpha_n &= \argmax\limits_\alpha\{L^\alpha u_n - g^\alpha\}\\
u_{n+1} &= (L^{\alpha_n})^{-1}g^{\alpha_n}.
\end{align*}

Computed results are presented in Table~\ref{table:CE}.  Despite the very low regularity of this example, the method converges, with a rate that appears close to the formal discretisation error of $\bO(\sqrt{h})$.

\begin{figure}[htp]
\centering
{\subfigure[]{\includegraphics[width=0.4\textwidth]{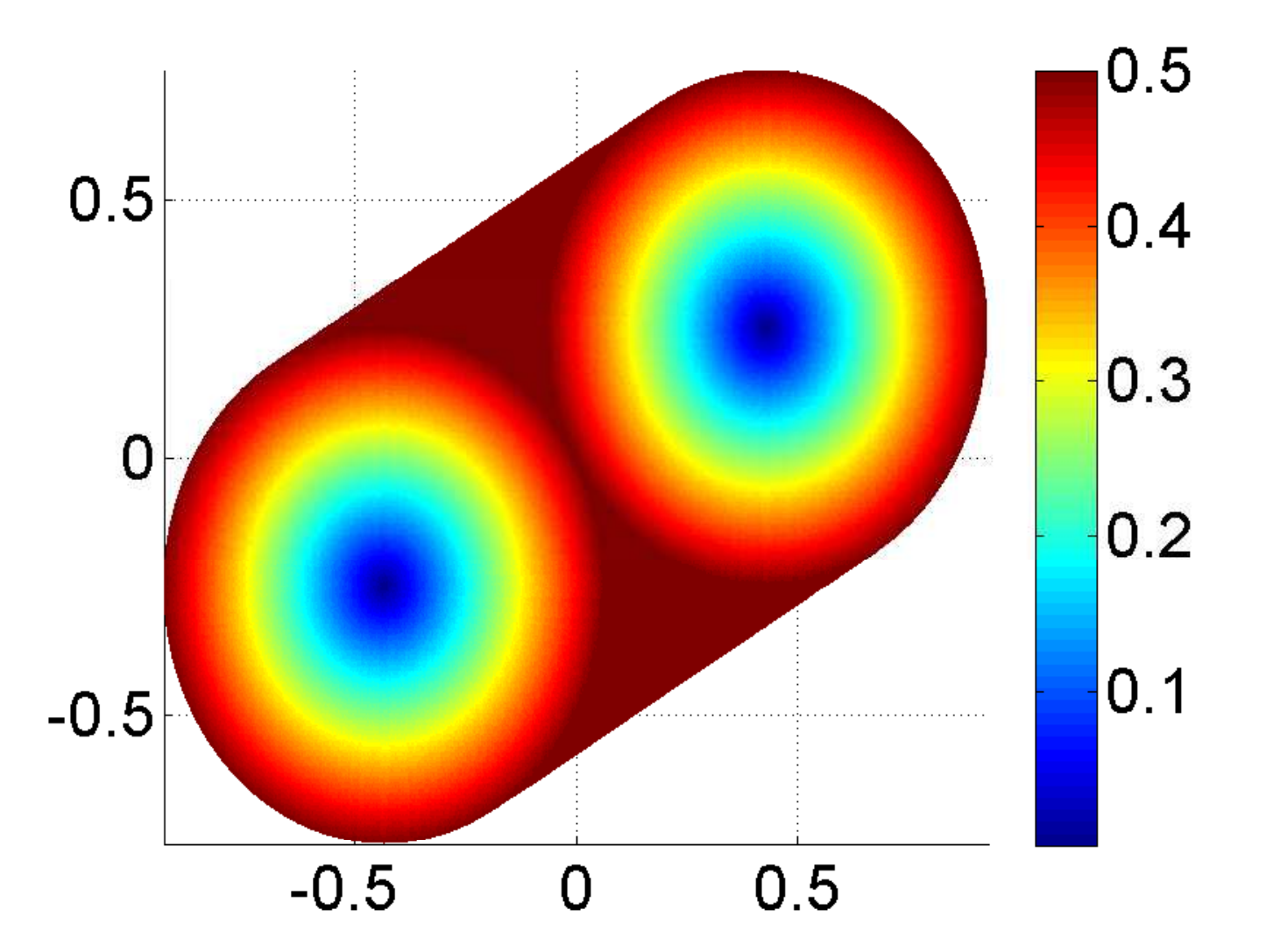}\label{fig:CEObs}}}
{\subfigure[]{\includegraphics[width=0.4\textwidth]{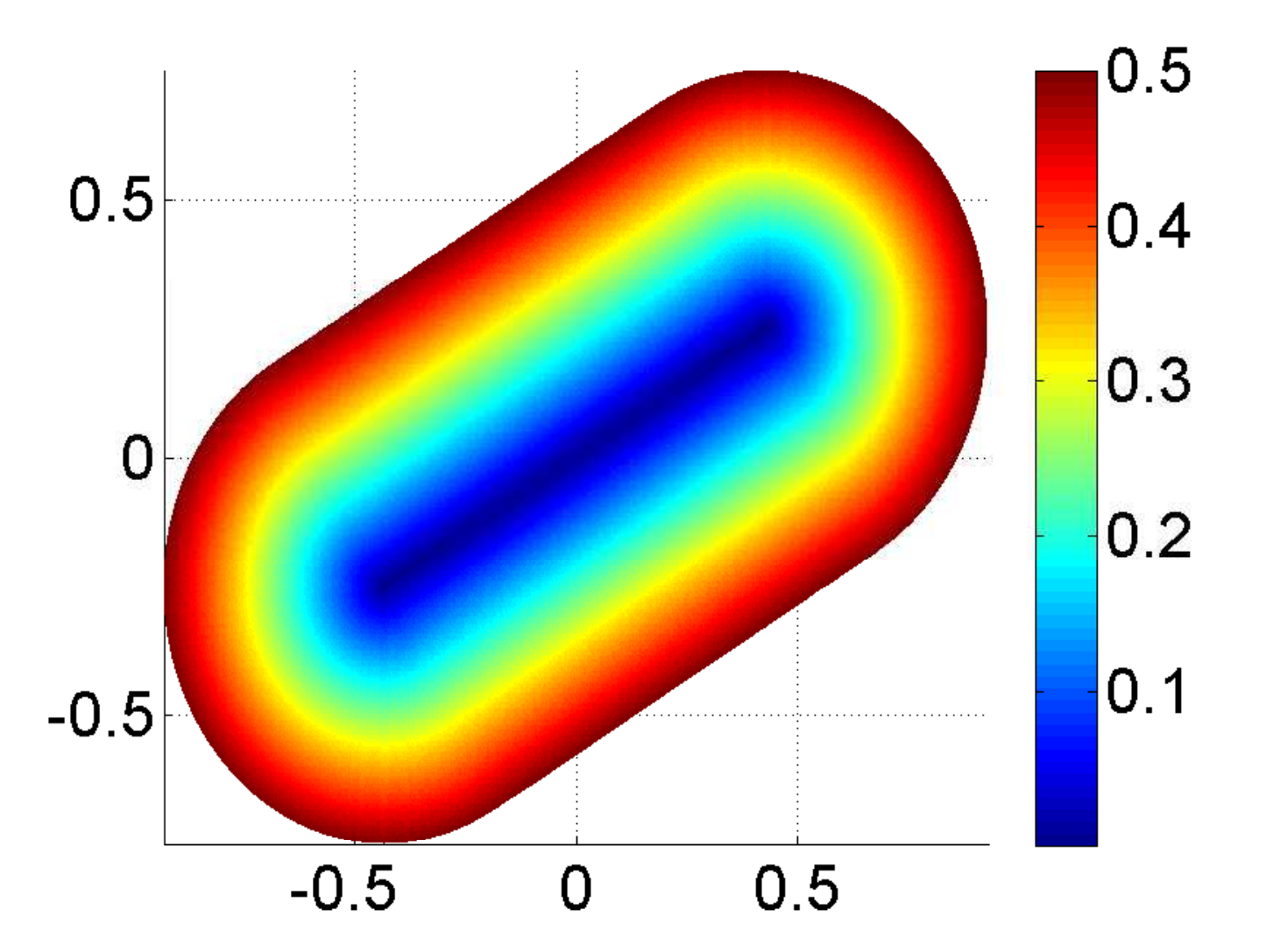}\label{fig:CESol}}}
\caption{\subref{fig:CEObs}~Obstacle~$g$ and \subref{fig:CESol}~solution~$u$ of the convex envelope equation~\eqref{eq:CE}.}
\label{fig:CE}
\end{figure}

\begin{table}[htp]
\centering
\begin{tabular}{ccccc}
$h$ & $N$ & Max Error & Rate ($h$) & Rate ($N$) \\
\hline
2/32  & 1,191  & $3.9\times10^{-2}$ & --- & --- \\
2/64  & 3,873  & $5.9\times10^{-2}$ & -0.6 & -0.3\\
2/128 & 13,069 & $2.7\times10^{-2}$ & 1.1 & 0.6\\
2/256 & 45,529 & $1.5\times10^{-2}$ & 0.9 & 0.5\\
2/512 & 163,081& $1.1\times10^{-2}$ & 0.4 & 0.2
\end{tabular}
\caption{Convergence results for the convex envelope equation~\eqref{eq:CE}.}
\label{table:CE}
\end{table}

\subsection{Obstacle problem}\label{sec:obstacle}

In our next example, we demonstrate the ease with which our meshfree approximations can be used on complicated domains.  To do this, we solve the obstacle problem
\bq\label{eq:obstacle}
\begin{cases}
\min\left\{-\Delta u, u-g\right\} = 0 , & x \in \Omega\\
u = 0, & x \in \partial\Omega
\end{cases}
\eq
on a domain $\Omega$ that contains both an exterior boundary and a highly-detailed interior boundary.
The point cloud (obtained from~\cite{dolfin}), obstacle~$g$, and computed solution~$u$ are presented in Figure~\ref{fig:dolphin}.  The nonlinear algebraic system was solved using policy iteration as in the previous example.

\begin{figure}[htp]
\centering
{\subfigure[]{\includegraphics[width=0.3\textwidth]{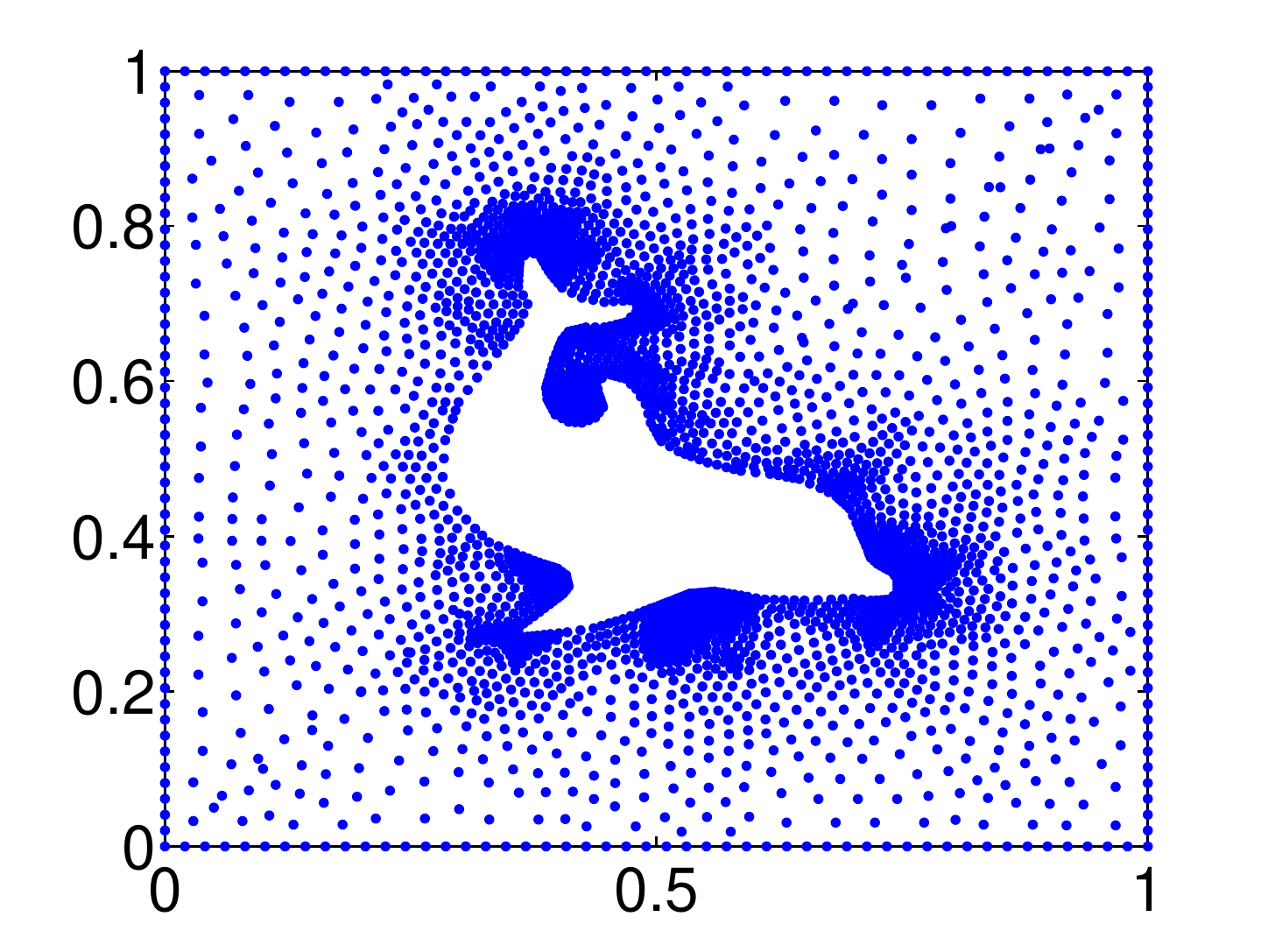}\label{fig:domain}}}
{\subfigure[]{\includegraphics[width=0.3\textwidth]{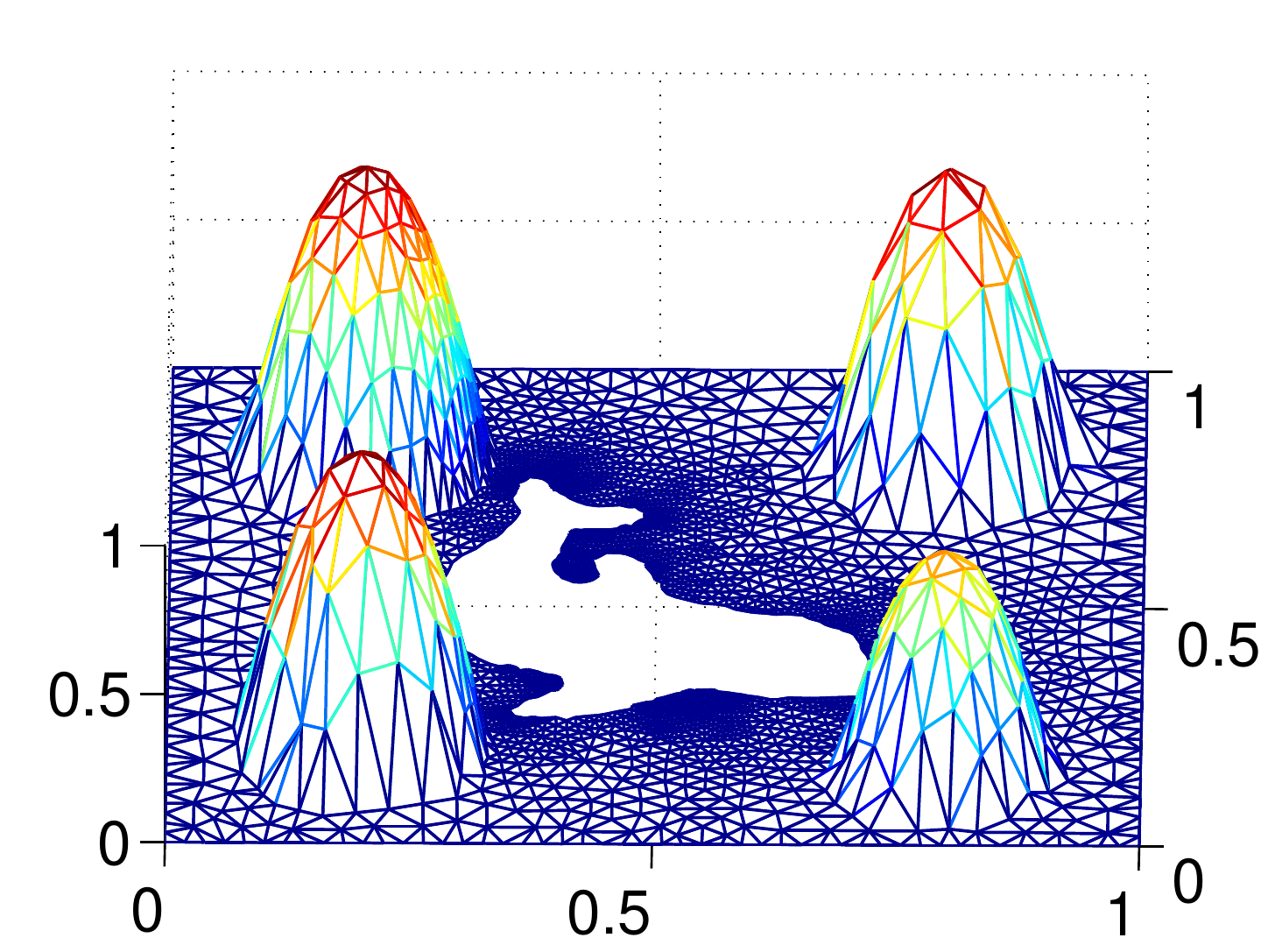}\label{fig:obstacle}}}
{\subfigure[]{\includegraphics[width=0.3\textwidth]{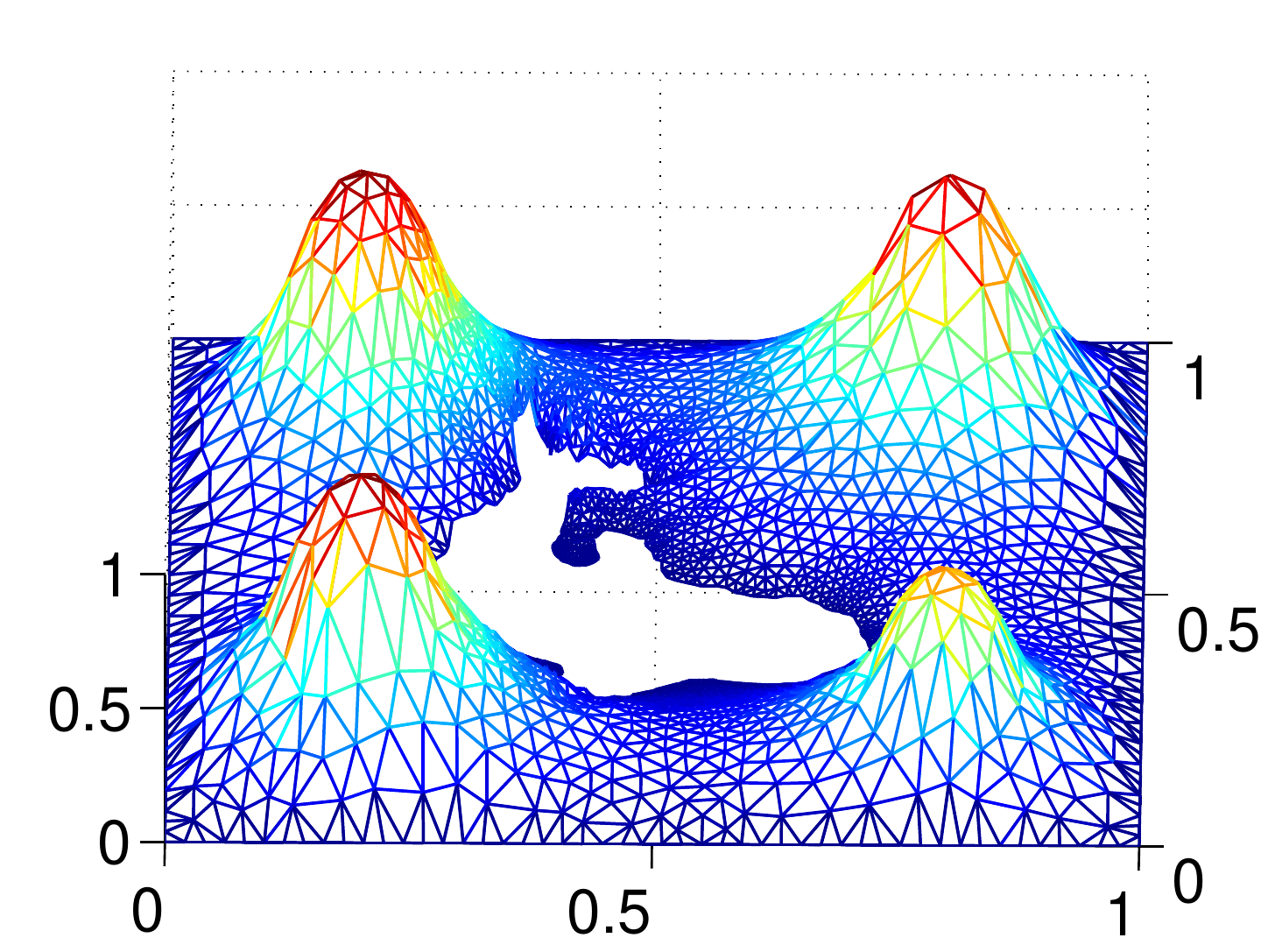}\label{fig:solution}}}
\caption{\subref{fig:domain}~A point cloud, \subref{fig:obstacle}~obstacle~$g$, and \subref{fig:solution}~computed solution~$u$ for the obstacle problem~\eqref{eq:obstacle}.}
\label{fig:dolphin}
\end{figure}

\subsection{Monge-Amp\`ere equation}\label{sec:MA}

For our final example, we consider the \MA equation
\bq\label{eq:MA}
\begin{cases}
-\det(D^2u) + f = 0, & x\in\Omega\\
u = g, & x\in\partial\Omega\\
u \text{ is convex.}
\end{cases}
\eq

This PDE is elliptic only on the space of convex functions.  However, as in~\cite{FroeseTransport}, we can make use of the globally elliptic extension
\begin{multline*}
-\min\limits_{\theta\in[0,\pi/2)}\left\{\max\left\{\frac{\partial^2u}{\partial e_\theta^2},0\right\}\max\left\{\frac{\partial^2u}{\partial e_{\theta+\pi/2}^2},0\right\} \right. \\ \left. +\min\left\{\frac{\partial^2u}{\partial e_\theta^2},0\right\}+\min\left\{\frac{\partial^2u}{\partial e_{\theta+\pi/2}^2},0\right\}\right\} + f = 0. 
\end{multline*}

As with the approximations of the eigenvalues, this minimum is approximated using derivatives in finitely many ($\sim\pi/d\theta$) directions.  We let the domain $\Omega$ be an ellipse with semi-major axis of length one and semi-minor axis of length $1/\sqrt{2}$.  Computations are performed on a uniform point cloud augmented by a uniform discretisation of the boundary.  The nonlinear systems were solving using a damped Newton's method as in~\cite{FO_MATheory}.

We consider two examples: a $C^2$ solution defined by
\[ u(x,y) =  e^{x^2+y^2}, \quad f(x,y) = (1+x^2+y^2)e^{x^2+y^2}\]
and a $C^1$ solution for which the ellipticity is degenerate in an open set,
\[ u(x,y) = \frac{1}{2}\max\left\{\sqrt{x^2+y^2}-0.2,0\right\}^2, \quad f(x,y) = \max\left\{1- \frac{0.2}{\sqrt{x^2+y^2}},0\right\}. \]
These functions are displayed in Figure~\ref{fig:MA}.

We begin with the smooth example.  Table~\ref{table:MAC2} indicates that the approximations converge, but as expected for this monotone scheme, the order of convergence is low.  

This situation can be improved by using the monotone scheme as the foundation for a higher-order filtered scheme of the form of~\eqref{eq:fdfilter}.  To do this, we use a second-order accurate finite difference approximation $F_A$ of
\[ -(u_{xx}u_{yy}-u_{xy}^2), \]
which is defined on the same (uniform) point cloud.  As discussed in subsection~\ref{sec:filter}, the formal discretisation error is independent of the size of the angular resolution~$d\theta$.  We take advantage of this fact and choose a larger angular resolution of $d\theta = 2h^{1/3}$.  This allows for a smaller search radius~$r$ and a lower boundary resolution~$h_B$. The results for this filtered scheme are also displayed in Table~\ref{table:MAC2}, which demonstrates that the filtered method is both less expensive and significantly more accurate.  In particular, for a given spatial resolution~$h$, fewer discretisation points are needed (because of the reduced boundary resolution), and the observed accuracy is second-order in~$h$.

We use the same filtered method to compute the~$C^1$ solution.  This solution is not classical and the ellipticity is degenerate; Newton's method applied to the non-monotone scheme on its own is not stable.  However, by filtering with the monotone scheme, we are able to obtain first-order convergence in~$h$.

\begin{figure}[htp]
\centering
{\subfigure[]{\includegraphics[width=0.45\textwidth]{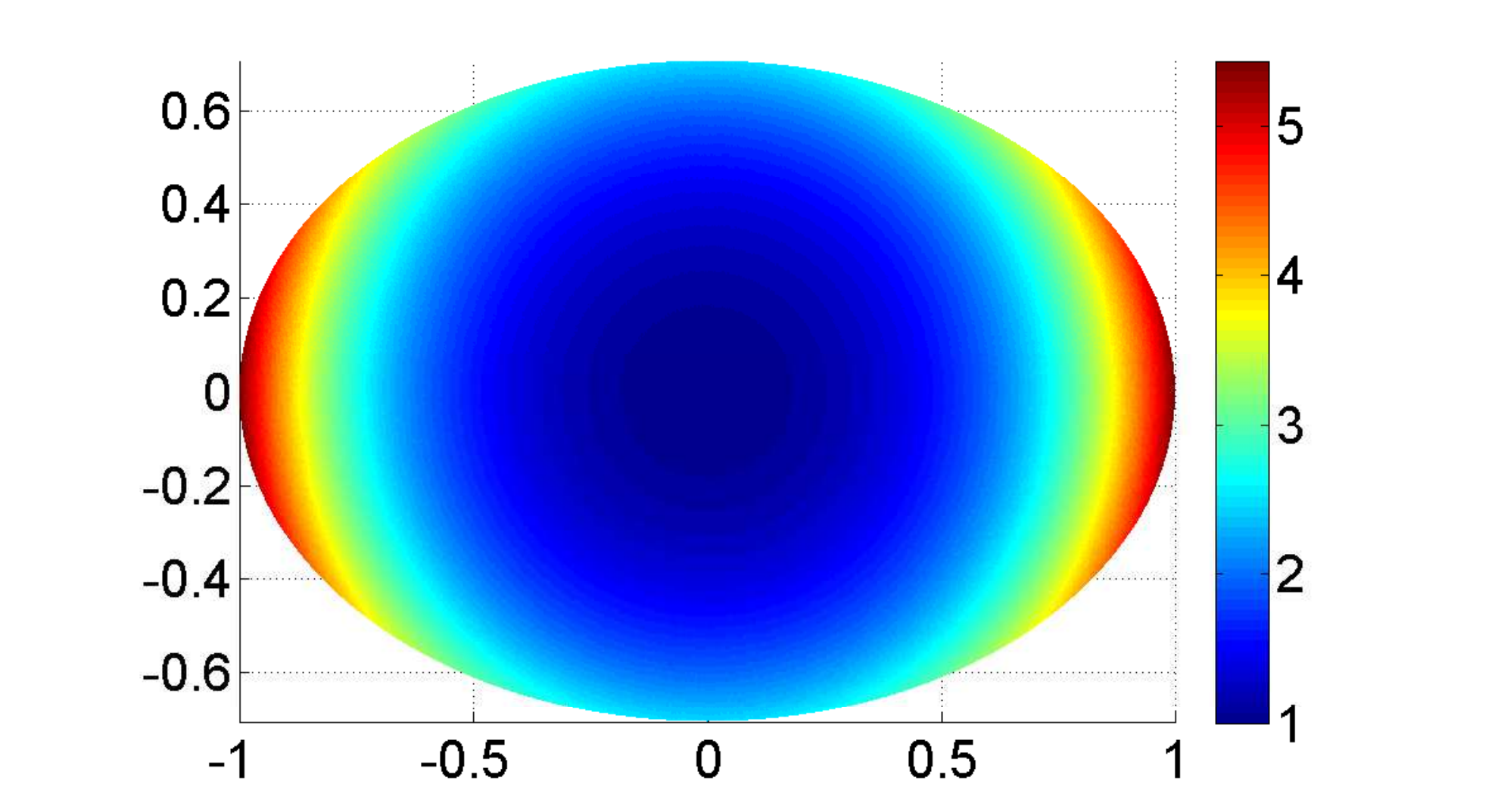}\label{fig:MA_fC2}}}
{\subfigure[]{\includegraphics[width=0.45\textwidth]{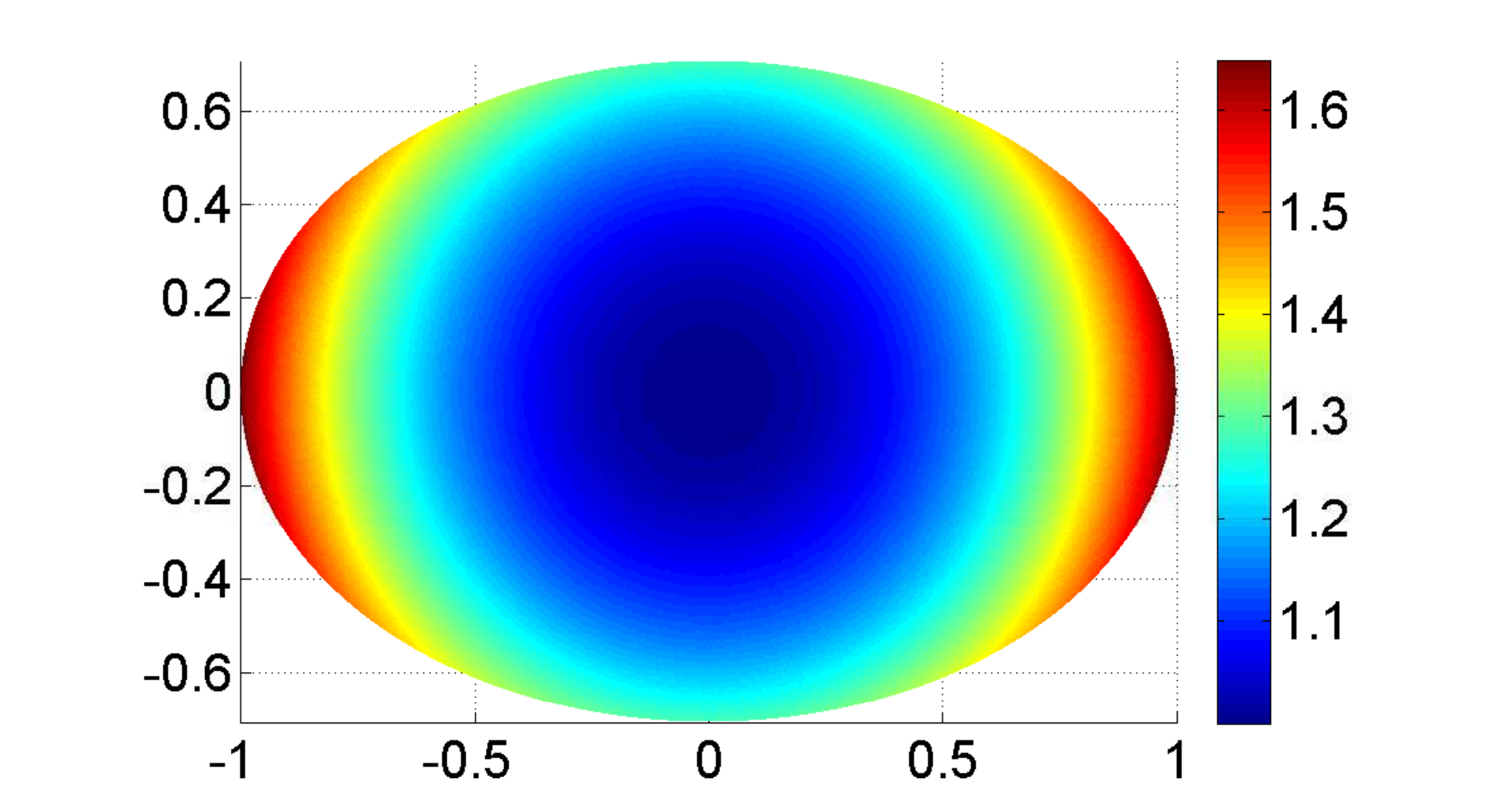}\label{fig:MA_uC2}}}
{\subfigure[]{\includegraphics[width=0.45\textwidth]{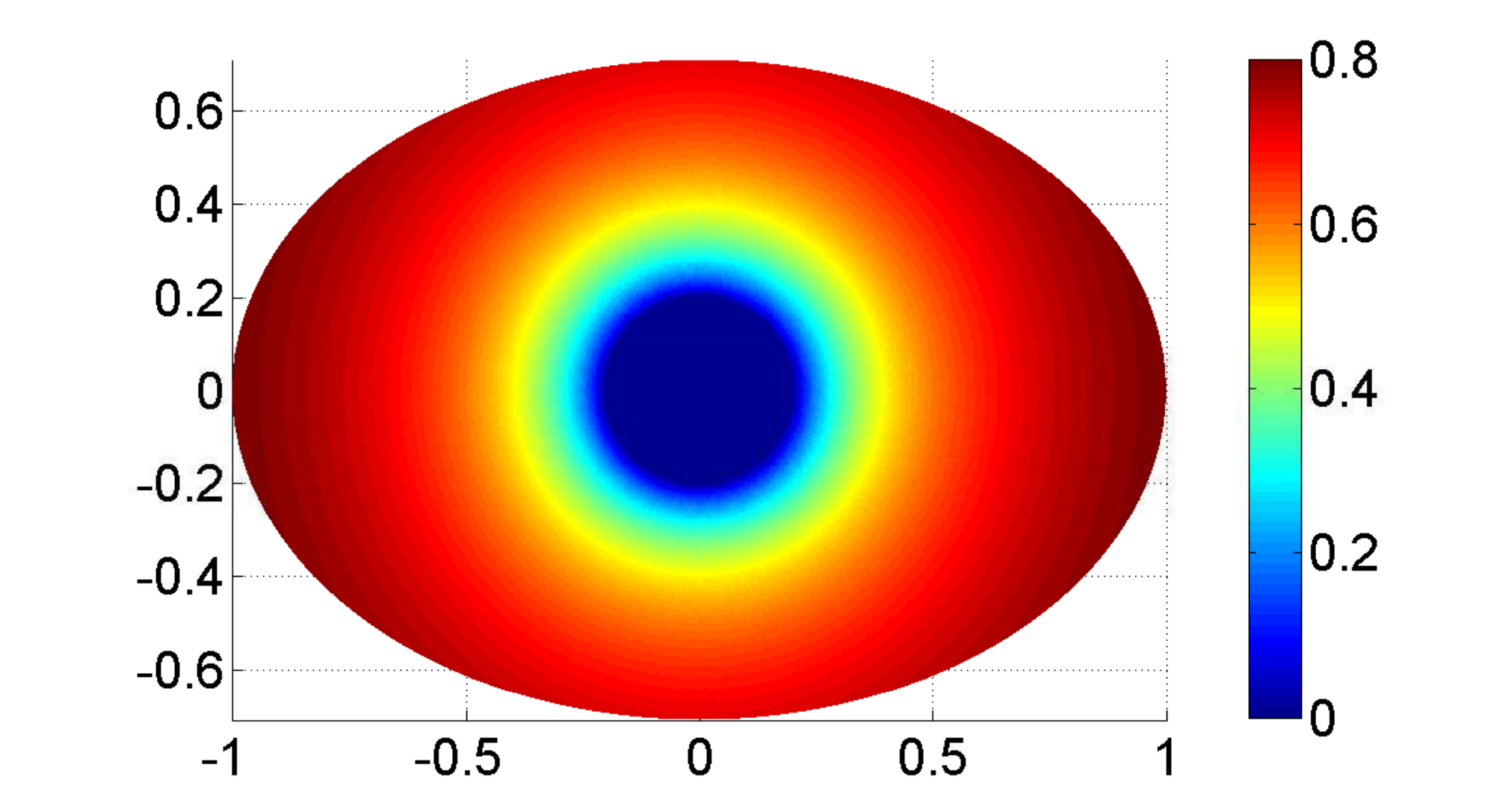}\label{fig:MA_fC1}}}
{\subfigure[]{\includegraphics[width=0.45\textwidth]{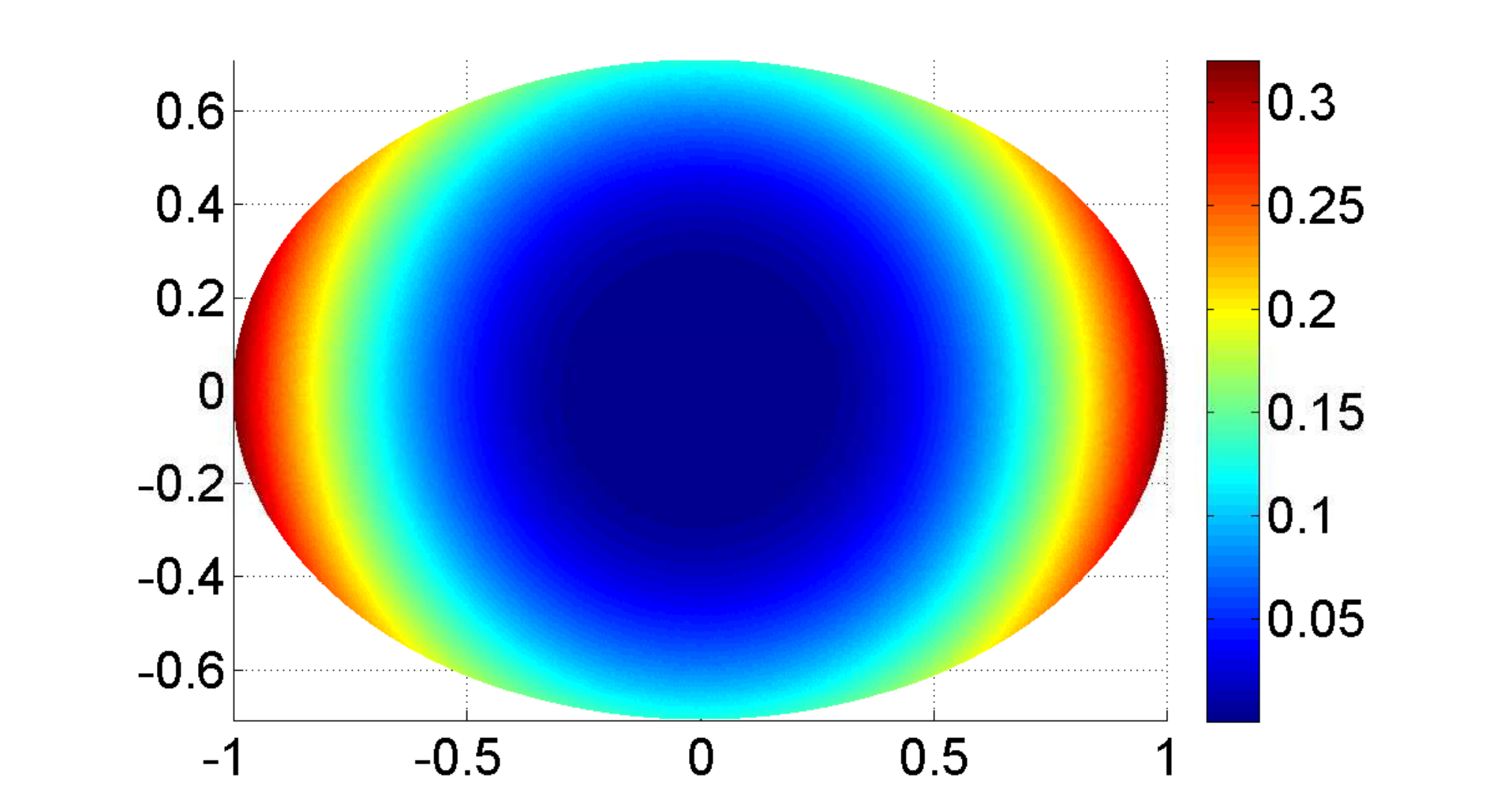}\label{fig:MA_uC1}}}
\caption{Right-hand side $f$ and solution $u$ for \subref{fig:MA_fC2},\subref{fig:MA_uC2}~$C^2$ and \subref{fig:MA_fC1},\subref{fig:MA_uC1}~$C^1$ solutions of the \MA equation~\eqref{eq:MA}.}
\label{fig:MA}
\end{figure}

\begin{table}[htp]
\centering
\begin{tabular}{c|ccc|ccc}
 & \multicolumn{3}{|c|}{Monotone} & \multicolumn{3}{c}{Filtered}\\
$h$ & $N$ & Max Error & Rate ($h$) & $N$ & Max Error & Rate ($h$)  \\
\hline
2/32  & 1,280  & $1.0\times10^{-3}$ & --- &945    & $7.8\times10^{-4}$ & ---\\
2/64  & 4,298  & $4.1\times10^{-4}$ & 1.33&3,247  & $9.0\times10^{-5}$ & 3.11\\
2/128 & 14,799 & $3.9\times10^{-4}$ & 0.09&11,545 & $2.6\times10^{-5}$ & 1.81 \\
2/256 & 52,590 & $2.8\times10^{-4}$ & 0.49&42,646 & $6.1\times10^{-6}$ & 2.07\\
2/512 & 191,467& $1.7\times10^{-4}$ & 0.72&161,417& $1.5\times10^{-6}$ & 2.01
\end{tabular}
\caption{Convergence results for a $C^2$ solution of the \MA equation~\eqref{eq:MA}.}
\label{table:MAC2}
\end{table}

\begin{table}[htp]
\centering
\begin{tabular}{cccc}
$h$ & $N$ & Max Error & Rate ($h$)  \\
\hline
2/32  & 945  & $3.0\times10^{-3}$ & --- \\
2/64  & 3,247  & $1.3\times10^{-3}$ & 1.17 \\
2/128 & 11,545 & $4.1\times10^{-4}$ & 1.71\\
2/256 & 42,646 & $1.7\times10^{-4}$ & 1.26 \\
2/512 & 161,417& $8.1\times10^{-5}$ & 1.07 
\end{tabular}
\caption{Convergence results for a $C^1$ solution of the \MA equation~\eqref{eq:MA}.}
\label{table:MAC1}
\end{table}

\section{Conclusions}\label{sec:conclusions}

We introduced new monotone meshfree finite difference methods for solving elliptic equations that depend on either the eigenvalues of the Hessian or other second directional derivatives.  The key to accomplishing this is to select finite difference stencils that align as closely as possible with the direction of interest, which can be accomplished as long as the search neighbourhood is sufficiently large relative to the resolution of the point cloud. These schemes are monotone, and we proved that they converge to the viscosity solution of the associated PDE.  They can also serve as the foundation for provably convergent higher-order filtered methods.

The methods were implemented and tested on a degenerate linear elliptic equation, the convex envelope equation, an obstacle problem, and the \MA equation.  Numerical tests demonstrated convergence on highly unstructured (eg random) point clouds, complicated domains, degenerate examples, and problems where the solution is only Lipschitz continuous.

Future work will extend these ideas to three dimensions and develop local criteria for the search neighbourhoods in order to improve the benefits of adaptivity.

\bibliographystyle{plain}
\bibliography{Meshfree}

\end{document}